%% file: paper.tex
\documentclass[11pt]{article}
\usepackage[a4paper,margin=1.0in]{geometry}
\usepackage{parskip}
\usepackage[T1]{fontenc}
\usepackage[utf8]{inputenc}
\usepackage[english]{babel}
\usepackage{amsmath,amssymb,amsthm,mathtools}
\usepackage{mathrsfs} 
\usepackage{lmodern}
\usepackage{microtype}
\mathtoolsset{showonlyrefs}
\usepackage{euscript}
\usepackage[dvipsnames, svgnames]{xcolor} 
\usepackage[shortlabels]{enumitem} 
\usepackage{bm}
\usepackage[misc]{ifsym}
\usepackage[normalem]{ulem}
\usepackage{setspace}
\usepackage{caption}
\usepackage{subcaption}
\usepackage{graphicx}
\usepackage[numbers, sort&compress]{natbib}
\definecolor{AirForceBlue}{RGB}{93,138,168}
\usepackage[unicode=true]{hyperref}
\usepackage[numbered]{bookmark}
\hypersetup{pdfencoding=auto, psdextra}
\hypersetup{
        colorlinks=false,
        pdfborder={0 0 1},
        linkcolor=MediumBlue,
        citecolor=ForestGreen,
        urlcolor=AirForceBlue,
        linktoc=page,
        filecolor=AirForceBlue,
      }%

\newtheorem{theorem}{Theorem}[section]
\newtheorem{corollary}[theorem]{Corollary}
\newtheorem{lemma}[theorem]{Lemma}

\newtheorem{proposition}[theorem]{Proposition}

\newtheorem{assumption}[theorem]{Assumption}
\newtheorem{assumptions}[theorem]{Assumptions}

\theoremstyle{definition}
\newtheorem{definition}[theorem]{Definition}
\newtheorem{algorithm}[theorem]{Algorithm}
\newtheorem{remark}[theorem]{Remark}

\numberwithin{equation}{section}

\input{latex_settings}

\renewcommand{\bm}{}

\providecommand{\norm}[1]{\lVert#1\rVert}
\providecommand{\scalarp}[1]{\langle#1\rangle}

\DeclareMathOperator*{\argmin}{arg\,min}

\DeclareMathOperator*{\dom}{dom}
\DeclareMathOperator*{\interior}{int}

\newcommand{\minimize}[2]{%
  \operatorname*{minimize}_{#1}\; #2%
}

\newcommand{\Fk}{\ensuremath{\mathfrak F}}
\newcommand{\EE}{\ensuremath{\mathsf E}}

\newcommand{\PP}{\ensuremath{\mathsf P}}
\newcommand{\R}{\ensuremath \mathbb{R}}
\newcommand{\HH}{\ensuremath \mathsf{H}}
\newcommand{\bxs}{\ensuremath \bm{\mathsf{x}}}
\newcommand{\bus}{\ensuremath \bm{\mathsf{u}}}
\newcommand{\bzs}{\ensuremath \bm{\mathsf{z}}}

\newcommand{\br}{\ensuremath \bm{{r}}}

\newcommand{\xx}{\ensuremath \mathsf{x}}

\newcommand{\N}{\ensuremath \mathbb{N}}

\newcommand{\gs}{\ensuremath{{\mathsf{g}}}}

\newcommand{\bys}{\ensuremath \bm{\mathsf{y}}}

\newcommand{\bx}{\ensuremath \bm{x}}
\newcommand{\by}{\ensuremath \bm{y}}
\newcommand{\bv}{\ensuremath \bm{v}}
\newcommand{\bw}{\ensuremath \bm{w}}
\newcommand{\bvs}{\ensuremath \bm{\mathsf{v}}}

\newcommand{\be}{\ensuremath \bm{e}}
\newcommand{\bg}{\ensuremath \bm{g}}
\newcommand{\bu}{\ensuremath \bm{u}}
\newcommand{\bz}{\ensuremath \bm{z}}

\newcommand{\bgs}{\ensuremath{{\gs}}}

\newcommand{\bvarphi}{\ensuremath{{\varphi}}}

\newcommand{\bbar}{\bm \bar}

\newcommand{\bphi}{\ensuremath \bm{\phi}}
\newcommand{\alpakf}{\ensuremath \alpha_k^2 A_k}
\newcommand{\alpakff}{\ensuremath \alpha_k A_k}
\newcommand{\alpaks}{\ensuremath \alpha_k^2 B_k}
\newcommand{\ck}{\ensuremath C_k}

\renewcommand{\refname}{References}
\newcommand{\HHs}{\mathcal{C}}

\newcommand{\myref}[1]{%
  \mbox{\ref{#1}\vphantom{()}}%
}
\makeatletter
\DeclareRobustCommand{\myhyperref}[3]{%
  \mbox{%
    \hyperref[#1]{%
      \ref*{#2}\,\textup{\tagform@{\ref*{#3}}}\,%
    }\vphantom{()}%
  }%
}
\makeatother
\DeclareRobustCommand{\myhyperrefrange}[3]{%
  \mbox{%
    (\hyperref[#1]{%
      \ref*{#2}\,--\,\ref*{#3}%
    })\vphantom{()}%
  }%
}
\makeatletter
\DeclareRobustCommand{\myeqhyperrefrange}[3]{%
  \mbox{%
    \hyperref[#1]{%
      \,\textup{\tagform@{\ref*{#2}}}\,--\,\textup{\tagform@{\ref*{#3}}}\,%
    }\vphantom{()}%
  }%
}
\makeatother

\newcommand{\myEqref}[1]{\mbox{\hyperref[#1]{(\ref*{#1})}\vphantom{()}}}
\makeatletter
\NewDocumentCommand{\runinsectionstar}{o m}{%
  \IfNoValueTF{#1}{%
    \def\runin@size{\Large}%
  }{%
    \def\runin@size{#1}%
    \if\relax\detokenize{#1}\relax
      \def\runin@size{\Large}%
    \fi
  }%
  \addcontentsline{toc}{section}{#2}%
  \@startsection{section}%
    {1}%
    {\z@}%
    {-3.5ex \@plus -1ex \@minus -.2ex}%
    {-1em}%
    {\normalfont\runin@size\bfseries}*%
    {#2.}%
}
\makeatother
\makeatletter
\NewDocumentCommand{\runinsubsectionstar}{o m}{%
  \IfNoValueTF{#1}{%
    \def\runin@size{\large}%
  }{%
    \def\runin@size{#1}%
    \if\relax\detokenize{#1}\relax
      \def\runin@size{\large}%
    \fi
  }%
  \addcontentsline{toc}{section}{#2}%
  \@startsection{subsection}%
    {2}%
    {\z@}%
    {-3.25ex \@plus -1ex \@minus -.2ex}%
    {-1em}%
    {\normalfont\runin@size\bfseries}*%
    {#2.}%
}
\makeatother

\title{Bregman Stochastic Proximal Point Algorithm with Variance Reduction}
\author{
Cheik Traor\'e$\,^{\textrm{\Letter}}\,$\thanks{Corresponding author: Toulouse School of Economics, Toulouse Capitole University, Toulouse, France (\texttt{cheik.traore@tse-fr.eu}).}
\and
Peter Ochs\thanks{Department of Mathematics and Computer Science, Saarland University, Saarbr\"ucken, Germany (\texttt{ochs@cs.uni-saarland.de}).}
}

\date{}

\begin{document}
\maketitle
\begin{abstract}
    Stochastic algorithms, in particular stochastic gradient descent (SGD), have become the methods of choice in data science and machine learning. More recently, the stochastic proximal point algorithm (SPPA) has been introduced and shown to be more robust than SGD with respect to stepsize selection. However, SPPA still suffers from reduced convergence rates due to the need for vanishing stepsizes, a limitation that can be alleviated by variance reduction techniques. In the deterministic setting, many optimization problems can be solved more efficiently by exploiting non-Euclidean geometries through the use of Bregman distances. In this work, we bridge these two lines of research and propose variance reduction techniques for the Bregman stochastic proximal point algorithm (BSPPA). As special cases, our framework yields SAGA- and SVRG-type variance reduction methods for BSPPA. Our theoretical analysis and numerical experiments demonstrate improved stability and faster convergence rates compared to the vanilla BSPPA, both with constant and vanishing stepsizes, respectively. Moreover, our analysis also encompasses the SGD setting, allowing us to recover the same variance reduction techniques for Bregman SGD within a unified framework.
\end{abstract}

\vspace{1ex}

\section{{Introduction}}

The objective of the paper is to solve the following finite-sum optimization problem
\begin{align}
  \label{prob:main-intro}
    \minimize{ \bxs \in \HH}{ \frac{1}{n}\sum_{i = 1}^{n}f_i(\bxs)}, 
\end{align}
where
$\HH \subset \mathbb{R}^d$ is closed and convex with nonempty interior, and
$f_i\colon \mathbb{R}^d \to \R \cup \{+\infty\} $ is convex, for $i \in \{1, \ldots, n \}$.

The prime example of Problem~\eqref{prob:main-intro} is the Empirical Risk Minimization (ERM) problem in machine learning \citep[Section 2.2]{shalev2014understanding}. In that setting, $n$ is the number of data points, $\bxs\in\mathbb{R}^d$ includes the parameters of a machine learning model (linear functions, neural networks, etc.), and the function $f_i$ is the loss of the model $\bxs$ at the $i$-th data point. To efficiently solve Problem~\eqref{prob:main-intro}, stochastic approximation \citep{robbins1951stochastic} is leveraged, in particular stochastic gradient descent (SGD) and its variants \citep{duchi2011adaptive, kingma2014adam}.

\textbf{Stochastic Proximal Point Algorithm.} In recent years, the stochastic proximal point algorithm (SPPA) \citep{ryu2014stochastic,asi2019stochastic, kim2022convergence, bertsekas2011incremental,patrascu2017nonasymptotic,toulis2016towards, toulis2017asymptotic, toulis2015proximal} has emerged as a good alternative to SGD, demonstrating greater robustness to stepsize selection. Instead of the gradient $\nabla f_i$, the proximity operator of $f_i$, chosen randomly, is used at each iteration. 
\citet{davis2018stochastic} studied a Bregman distance version of SPPA (BSPPA) to account for a non-Euclidean geometry that can be adapted to the problem. While the convergence rates of both SPPA and BSPPA are of the same orders as those of SGD, in the deterministic setting, well adapted Bregman based algorithms can improve the constants in the rates significantly (e.g. from linear to logarithmic dependence on the problem dimension \citep{NY83}). We trace this dis-balance between the stochastic and deterministic setting back to the variance of the stochastic estimator of the true proximity operator that requires vanishing stepsize for the algorithm to converge. 

{\textbf{Variance-reduced Stochastic Proximal Point Algorithms} have only recently emerged  \citep{traore2024variance, khaled2022faster, milzarek2022semismooth, richtarik2025a}. As in this paper, all existing convergence results are provided in the smooth (differentiable) case, except for Point-SAGA proposed by \citet{defazio2016simple} and its generalization SMPM proposed by \citet{condat2025stochastic}, where sublinear convergence rates ($\mathcal{O}(1/k)$ and $\mathcal{O}(1/k^2)$ respectively) with a constant stepsize are provided for strongly convex functions.} However, to the best of our knowledge, no existing result has been established in the Bregman setting. The latter can be better adapted to constrained optimization problems or to cases where the Euclidean distance fails to adequately capture the underlying properties.

\textbf{Contributions.} In this paper we propose a Bregman Stochastic Proximal Point Algorithm (BSPPA) with generic variance reduction. We provide improved convergence rates as compared with vanilla BSPPA without variance reduction, in particular, sublinear and linear rates for convex or relatively strongly convex functions, respectively. Then, from the generic results, convergence rates for BSAPA, BSVRP, and L-BSVRP, which are proximal and Bregman distance versions of SAGA \citep{defazio2014saga}, SVRG \citep{johnson2013accelerating}, and L-SVRG \citep{kovalev2020don} respectively. These variance-reduced algorithms are all new in the literature. We can also recover rates for vanilla BSPPA for potentially nonsmooth functions. Of course, the rates of the variance-reduced versions (BSAPA, BSVRP, L-BSVRP) are given without the need for vanishing stepsizes, and are faster than that of BSPPA, which are the intended objectives of variance reduction. As a by-product of our study, we provide a unified variance reduction analysis for Bregman SGD. Our analysis can thus recover not only the Bregman version of SAGA and L-SVRG (present in \citet{pmlr-v139-dragomir21a}), but also Bregman SVRG and the Bregman version of the unified study in \citet{gorbunov2020unified}, dealing with several variants of (variance-reduced) Euclidean SGD.

\section{{Preliminaries}}

\subsection{Notation}
Let $\mathcal{C} \subset \mathbb{R}^d$ be convex with nonempty interior. We denote interior of $\mathcal{C}$ by $\interior\mathcal{C}$, its closure by $\bbar{\mathcal{C}}$ and its boundary by \(\text{bd}\, \mathcal{C}\).
For every integer $\ell \geq 1$, we define $[\ell] := \{1, \dots, \ell\}$.  
Bold default font is used for random variables taking values in $\R^d$, while bold sans serif font is used for their realizations or deterministic variables in $\R^d$. The probability space underlying random variables is denoted by $(\Omega, \mathfrak{A}, \PP)$. For every random variable $\bx$, $\EE[\bx]$ denotes its expectation, while if $\Fk \subset \mathfrak{A}$ is a sub $\sigma$-algebra we denote by $\EE[\bx\,\vert\, \Fk]$ the conditional expectation of $\bx$ given $\Fk$. Also, $\sigma(\by)$ represents the $\sigma$-algebra generated by the random variable $\by$. For a function $\bvarphi\colon \R^d \to \R \cup \{+\infty\}$, define $\dom \bvarphi \coloneqq \left\{\bxs \in \R^d: \bvarphi(\bxs) < +\infty\right\}$. $\bvarphi$ is proper if $\dom \varphi \neq \varnothing$. The set of minimizers of $\bvarphi$ is $\argmin \bvarphi 
= \{\bxs \in \R^d : \bvarphi(\bxs) = \inf \bvarphi\}$. If $\inf \bvarphi$ is finite, it is represented by $\bvarphi_*$. 
We denote the subdifferential of $\varphi$ at $\bxs$ as $\partial\varphi(\bxs)\coloneqq \left\{\bgs \in \R^d: \bvarphi(\bys) - \bvarphi(\bxs) \geq \langle \bgs, \bys - \bxs \rangle\;(\forall\,\bys \in \R^d)\right\}$. When $\bvarphi$ is differentiable $\nabla \bvarphi$ denotes the gradient of $\bvarphi$. \(\bvarphi\) is said \(L\)-smooth if its gradient is \(L\)-Lipschitz continuous.
The (convex) conjugate of $\bvarphi$ is the function $\bvarphi^* \colon \R^d \to [-\infty, + \infty]$ defined by $ \bys \mapsto \sup_{\bxs \in \R^d} \langle \bys, \bxs \rangle - \bvarphi(\bxs).$ 
In this work, $\ell^1$ represents the space of sequences with summable norms and $\ell^2$ 
with summable squared norms. We set \(\R^d_{+} = \{\bxs \in \R^d: \xx_i \geq 0, i
\in [d]\}\) and \(\R^d_{++} = \{\bxs \in \R^d: \xx_i > 0, i
\in [d]\}\).
\subsection{Bregman distance}
The Bregman distance $D_h(\bxs, \bys)$, between $\bxs \in \mathcal{C}, \bys \in \operatorname{int} \mathcal{C}$, is defined as $ D_h(\bxs, \bys) \coloneqq h(\bxs) - h(\bys) - \langle \nabla h(\bys), \bxs - \bys \rangle$,
where $h \colon \R^d \to \R \cup \{+\infty\}$, often called kernel, is a strictly convex and twice continuously differentiable function on $\operatorname{int} \mathcal{C}$, with $\dom h = \mathcal{C}$ and $\displaystyle \lim_{\bxs \in \interior \mathcal{C}, \bxs \to \text{bd}\,C} \|\nabla h(x)\| = +\infty$. In what follows, $h$ will always refer to a function with those properties, unless stated otherwise. A classical example is the usual squared Euclidean distance $D_h(\bxs, \bys) = \frac{1}{2}\|\bxs - \bys \|_2^2$ when $h=\frac{1}{2}\|\cdot\|_2^2$. The Bregman distance often captures the underlying geometry of an optimization problem more effectively than the Euclidean distance. It can be better suited to the constraints set or to the objective function properties. This makes it a powerful tool in optimization. For instance, the Kullback--Leibler divergence on $\mathbb{R}^d$ does not have a global Lipschitz continuous gradient, however it is relatively smooth w.r.t. the kernel \(h(\bxs) = - \sum_i \log \xx_i\); see \citet{bauschke2017descent}. 
Unfortunately, the Bregman distance is not necessarily symmetric, homogeneous nor translation invariant. The latter two shortcomings are the reasons why the following assumption is required throughout the paper.
\begin{assumption}\label{ass:a14}
    Let $\bxs, \bys \in \interior \dom h^*$, 
    $\lambda \in \R$, 
    and $ \bzs \in \R^d$ such that $\bxs + \lambda \bzs, \bys + \bzs\in \dom h^*$. 
    There exists a positive gain function $G$ such that
    \begin{equation}\label{eq:gain}
        D_{h^*}(\bxs + \lambda \bzs, \bxs) \leq G(\bxs, \bys, \bzs ) \lambda^2 D_{h^*}(\bys + \bzs, \bys).
    \end{equation}
\end{assumption}
According to \citet{pmlr-v139-dragomir21a}, this assumption seems unavoidable when using past iterates in an algorithm with Bregman distance. This is also the case of accelerated methods where similar assumption is made; see \citet{hanzely2021accelerated}. The two papers provide extensive discussion and numerous examples regarding this assumption. Below, we present just one possible instance.
\begin{proposition} \emph{\citep[Proposition 1]{pmlr-v139-dragomir21a}.} If $h$ is $L$-smooth and the Hessian $\nabla^2 h^*$ is $M$-smooth, 
then the gain function can be chosen as:
\begin{equation*}
    G(\bxs,\bys,\bvs) = 1 + 2 M L \bigl( \|\bys - \bxs\| + \|\bvs\| \bigr).
\end{equation*}
\end{proposition}

To address the symmetry issue, as \citet{bauschke2017descent}, we will make use of the following symmetry coefficient. It will only be needed in the relatively strongly convex cases.
\begin{definition}
    Given a kernel $h\colon \R^d \to \R \cup \{+\infty\}$, its symmetry coefficient is defined by
 \begin{equation*}
     \gamma_h := \inf \Bigg\{ \frac{D_h(\bxs, \bys)}{D_h(\bys, \bxs)} \colon (\bxs, \bys) \in (\operatorname{int}\mathcal{C})^2, \bxs \neq \bys \Bigg\} \in [0, 1].
 \end{equation*}
\end{definition}
\begin{remark}
\begin{enumerate}
    \item From \citet[Theorem 3.7]{bauschke1997legendre}, we know that, for all $(\bxs, \bys)\in(\operatorname{int}\,\mathcal{C})^{2}$, $D_h(\bxs, \bys) = D_{h^*}(\nabla h(\bys), \nabla h(\bxs))$ . Hence, $\gamma_h = \gamma_{h^*}$.
    \item By definition, $(\forall \bxs \in \interior\mathcal{C})\;(\forall \bys \in \interior\mathcal{C}),$
    \begin{align}
    \gamma_h D_{h}(\bys, \bxs)\leq D_{h}(\bxs, \bys)\leq\gamma_h^{-1}D_{h}(\bys, \bxs),
    \end{align}
    where it is agreed that $0^{-1} =+\infty$ and $+\infty\times r = +\infty$ for all $r\geq 0$.
\end{enumerate}
\end{remark}
Next, we present a handy identity for Bregman distances.
\begin{lemma}[Three points identity \citep{chen1993convergence}]\label{lem:20250424a} Let $h\colon \R^d \rightarrow \R \cup \{+\infty\}$ be a proper lower semi-continuous convex function. For any $\bxs \in \dom h$, and $\bys, \bzs \in \interior\dom h$ the following identity holds:
\begin{equation*}
    D_h(\bxs, \bzs) - D_h(\bxs, \bys) - D_h(\bys, \bzs) = \langle \nabla h(\bys) - \nabla h(\bzs), \bxs - \bys \rangle.
\end{equation*}
\end{lemma}

\section{{Problem setting}}

First, we (re)introduce the finite-sum problem and collect all assumptions that we need for this paper, while not all of them will be required to hold at the same time. Section~\myref{sect:algo} presents the proposed algorithm for solving such problems:
\begin{align}
  \label{prob:main}
    \minimize{ \bxs \in \HH}{F(\bxs)},\quad  F(\bxs) := \frac{1}{n}\sum_{i = 1}^{n}f_i(\bxs)\,,
\end{align}
where  $\HH \coloneqq \bbar{\HHs}$ is the closure of $\HHs$ and
$f_i\colon \mathbb{R}^d \to \R \cup \{+\infty\} $ is proper, convex and lower semicontinuous (lsc), $i \in \{1,2, \ldots, n \}$.
\begin{assumptions}\label{ass:funtion}
\begin{enumerate}[label=\rm(A.\roman*),ref={\rm A.\roman*},leftmargin=*]
    \item \label{ass:a1} $\mathcal{C} \subset \dom F$ and $\bxs_* \in \mathcal{C}$ solves \eqref{prob:main}.
    \item \label{ass:a11} For all $i \in [n]$, $f_i$ is $\beta$-relatively strongly convex w.r.t. $h$, with $\beta \geq 0$, i.e.,
    \begin{equation}\label{PL condition}
        (\forall \bxs, \bys \in \interior\mathcal{C}) \quad \beta D_{_h} (\bxs, \bys) \leq  D_{f_i} (\bxs, \bys).
    \end{equation}
    It includes the standing convexity assumption of each $f_i$, i.e. $\beta = 0$.
     \item \label{ass:a13} For all $i \in [n]$, $f_i$ is $L$-relatively smooth w.r.t. to $h$, i.e., differentiable on $\interior \mathcal{C}$ and
    \begin{equation*}
        (\forall \bxs, \bys \in \interior\mathcal{C}) \quad D_{f_i} (\bxs, \bys) \leq L D_{h} (\bxs, \bys),
    \end{equation*}
    for some $L>0$.
    As a consequence, $F$ is $L$-relatively smooth.
    \item \label{ass:a12} $F$ is $\mu$-relatively strongly convex w.r.t. $h$, with $\mu > 0$.
    \item \label{ass:intmin}
    $\bxs_* \in \interior \mathcal{C}$ such that $\nabla F(\bxs_*) = 0$ and $h^*$ is full domain.
\end{enumerate}
\end{assumptions}
Assumption~\eqref{ass:a1} is always supposed true in this paper. 
Assumptions~\eqref{ass:a11} and~\eqref{ass:a12} extend the Euclidean notions of strong convexity and $L$-smoothness (i.e., $L$-Lipschitz continuity of the gradient), respectively, to the Bregman setting. Assumption~\eqref{ass:a11}, with $\beta > 0$, is only required when establishing a linear convergence rate (to a neighborhood of the optimum) for the vanilla BSPPA algorithm.
Similarly, \eqref{ass:a12}, which corresponds to the relative strong convexity of $F$, will be used to derive linear convergence rates to the exact minimum for the variance-reduced variants of BSPPA; in that case, only convexity of each $f_i$ will be assumed. Finally, \eqref{ass:a12} and~\eqref{ass:a11} with $\beta > 0$ are not required to obtain the sublinear convergence rates. Regarding Assumption \eqref{ass:intmin}, exactly like \citep{pmlr-v139-dragomir21a}, our analysis can not deal with the case where \(\bxs_* \in \text{bd}\,\mathcal{C}\). It would be a valuable extension to our work. 
\begin{section}{{BSPPA with generic variance reduction}}\label{sect:algo}
We propose to solve Problem~\eqref{prob:main} using the following Bregman stochastic proximal point algorithm (BSPPA) with a generic variance reduction satisfying the abstract conditions in Assumptions~\myref{ass:variance} further below. 
\begin{algorithm}
\label{algo:unified}
Let $(\be_k)_{k \in \N}$ be a sequence of random vectors in $\R^d$ and let $(i_k)_{k \in \N}$ be a sequence of i.i.d. random variables uniformly distributed on $\{1,\dots, n\}$, so that $i_k$ is independent of $\be_0$, \dots, $\be_{k-1}$. Let $\alpha_k >0$ and set $\bx_0 \equiv \bxs_0 \in \interior\mathcal{C}$. Then define, 
\begin{equation*}
\begin{array}{l}
\text{for}\;k=0,1,\ldots\\
\left\lfloor
\begin{array}{rl}
\bx_{k+1} &= \argmin_{\bxs \in \HH}\Big\{f_{i_{k}}(\bxs)-\scalarp{\be_k,\bxs-\bx_{k}} 
+\frac{1}{\alpha_k}D_{h}(\bxs,\bx_{k})\Big\}.
\end{array}
\right.
\end{array}
\end{equation*}
\end{algorithm}
We assume in all this work that this minimization has a unique solution in $\interior \mathcal{C}$. 
Compared to BSPPA, we have an additional linear perturbation that contains $\be_k$. It is the generic variance reduction term. As we shall see in Section~\myref{sect:applications}, depending on the specific algorithms, $\be_k$ may be defined in various ways. When $\be_k$ is set to $0$, BSPPA is recovered.
Using the optimality condition, the update at iteration $k$ can be rewritten as:
\begin{align}\label{eq:20220915c}
   \bx_{k+1} = \nabla h^* \Big(\nabla h (\bx_{k}) - \alpha_k\underbrace{(\bg_{k+1} - \be_k )}_{=: \bw_k} \Big),
\end{align}
where $\bg_{k+1} \in \partial f_{i_{k}}(\bx_{k+1})$ is such that~\eqref{eq:20220915c} holds. Since $\bw_k$ depends on $\bg_{k+1}$,  Equation~\eqref{eq:20220915c} shows that Algorithm~\myref{algo:unified} is an implicit algorithm. This is in contrast to an explicit update 
\begin{equation} \label {eq:10022025a}
  \bz_{k+1} = \nabla h^*\left( \nabla h (\bx_k) - \alpha_k \bv_k \right)
\end{equation}
in which $\bw_k$ is replaced by
\begin{equation}\label{eq:10022025c}
    \bv_k \coloneqq \bg_{k} - \be_k
\end{equation}
for a suitable $\bg_{k} \in \partial f_{i_{k}}(\bx_k)$. In minimization form, this explicit update reads
\begin{equation}\label{eq:sgdvr}
    \bz_{k+1} = \argmin_{\bxs \in \HH}\; \langle \bv_k, \bxs - \bx_k \rangle + \frac{1}{\alpha_k}D_{h}(\bxs,\bx_{k}),
\end{equation}
and is assumed well-posed.
\begin{remark}
The virtual explicit sequences $(\bv_{k})_{k\in\N}$ and $(\bz_{k})_{k\in\N}$ introduced in~\eqref{eq:10022025a} and~\eqref{eq:10022025c} are crucial to our analysis. They will be used to prove the main proposition even though they do not appear in Algorithm~\myref{algo:unified}.
\end{remark}
\begin{remark}[\textbf{Bregman variance-reduced SGD also covered}]
    The integration of these explicit iterates is the reason why our analysis extends seamlessly to the Bregman SGD case, providing a unified variance reduction study for Bregman SGD. All the variance reduction results will stand true for SGD with Bregman distance; see Remark~\myref{rmk:sgd} and Proposition~\myref{prop:unifiedexplicit}.
\end{remark}
The next assumptions that we consider concern the noise and variance of Algorithm~\myref{algo:unified}. The virtual explicit sequence $(\bz_{k})_{k\in\N}$ appears in those assumptions for the sake of analysis. This allows, at the same time, the application of the analysis and the results to the Bregman SGD case.
\begin{remark}\label{rk:assumproof}
    These assumptions on the noise and variance of the generic Algorithm~\myref{algo:unified} are ``assumed'' for the general analysis. But when algorithms are specified in Section~\myref{sect:applications} by defining $\be_k$, they will be proved by corresponding lemmas. They read as follows.
\end{remark}
\begin{assumptions}\label{ass:variance}
There exist non-negative sequences of real numbers $(A_k)_{k \in \N}, (B_k)_{k \in \N}, (C_k)_{k \in \N}$, $\rho \in [0,1]$, and a sequence of real-valued random variables $(N_k)_{k \in \N}$ such that, for every $k \in \N$,
\begin{enumerate}[label=\rm(B.\roman*),ref={\rm B.\roman*},leftmargin=*]
\item \label{ass:b1} $\EE[\be_k\,\vert\, \Fk_k] = 0$ and $\EE[\bg_k\,\vert\, \Fk_k] \in \partial F(\bx_k)$,
     \item \label{ass:b2} $ \EE[D_{h}(\bx_{k},\bz_{k+1}) \,\vert\, \Fk_k] \leq \alpakf \left (F(\bx_k) - F(\bxs_*)  \right) +\alpaks \sigma_k^2 + \alpha_k^2 N_{k}$,
    \item \label{ass:b3} 
    $\EE[\sigma_{k+1}^2] \leq (1-\rho) \EE\left[\sigma_k^2\right] + \ck\EE[F(\bx_k) - F(\bxs_*)]$,
\end{enumerate}
where $\sigma_k$ is a real-valued random variable (r.v.),
$(\Fk_k)_{k \in \N}$ is a sequence of $\sigma$-algebras such that, $\forall k \in \N$, 
$\Fk_k \subset \Fk_{k+1} \subset \mathfrak{A}$;
$i_{k-1}$, $\bx_k$, $\sigma_k^2$ and $N_k$  are $\Fk_k$-measurable, and 
$i_k$ is independent of $\Fk_k$.
\end{assumptions}
In the smooth case, Assumption~\eqref{ass:b1} ensures that $\EE[\bv_k \,\vert\, \Fk_k] = \EE[\nabla f_{i_k} (\bx_k)\,\vert\, \Fk_k] = \nabla F(\bx_k)$, so that the  direction $\bv_k$ is an unbiased estimator of the full gradient of $F$ at $\bx_k$, which is a standard assumption in the related literature. Assumption~\eqref{ass:b2} on $\EE\left[D_{h}(\bx_{k},\bz_{k+1}) \,\vert\, \Fk_k\right]$ is the equivalent of what is called, in the literature \citep{khaled2020better, demidovich2023guide} and in the Euclidean setting, the expected smoothness or $ABC$-assumption on $\EE\left[\norm{\nabla f_{i_k}(\bx_k)}^2 \,\vert\, \Fk_k\right]$ with  $\sigma_k = \|\nabla F(\bx_k)\|$ and $N_k$ constant.
Assumption~\eqref{ass:b3} provides some control on the variance from iteration to iteration. Indeed, as we will see later that the sequence $(\sigma_k)_{k \in \N}$ encodes the variance of the Algorithm~\myref{algo:unified}. Depending on $\rho$, \eqref{ass:b3} makes sure that the variance does not blow up and possibly reduces along the iterations whenever the algorithm converges.
\subsection{Convergence Analysis}\label{sect:analysis}
Now, we can present the main two theorems. 
\begin{theorem}[$F$ is only convex]\label{theo:unifiedmainonlyconvex}
Suppose that Assumptions~\myref{ass:variance} hold with $\rho > 0$. Let $(M_k)_{k\in \N}$ be a non-increasing positive real-valued sequence such that $M_k \geq \frac{B_k}{\rho}$, $\forall k \in \N$. Suppose also that the sequence $(\bx_k)_{k \in \N}$ is generated by Algorithm~\myref{algo:unified} with $(\alpha_k)_{k \in \N}$ a non-decreasing positive real-valued sequence such that $\alpha_k < \frac{1}{A_k + M_k C_k}$, $\forall k \in \N$. Then, for all $k \in \N$,
\begin{align}
    \EE[F(\bbar{\bx}_{k})- F(\bxs_*)] 
    &\leq \frac{(1/\alpha_{0}^2) \EE[ D_{h}(\bxs_*, \bx_{0})] + M_{0}\EE[\sigma_{0}^2]}{\sum_{t=0}^{k-1} (1/\alpha_{t})\left(1-\alpha_t(A_t + M_t C_t)\right)} 
    \nonumber \\
    &\qquad 
    + \sum_{t=0}^{k-1} \frac{\EE[N_t]}{\sum_{t=0}^{k-1} (1/\alpha_{t})\left(1-\alpha_t(A_t + M_t C_t)\right)}
     \nonumber,
\end{align}
with $\displaystyle \bbar{\bx}_k = \sum_{t=0}^{k-1} \frac{(1/\alpha_{t})\left(1-\alpha_t(A_t + M_t C_t)\right)}{\sum_{t=0}^{k-1} (1/\alpha_{t})\left(1-\alpha_t(A_t + M_t C_t)\right)} \bx_t$.
\end{theorem}
\begin{theorem}[F is $\mu$-relatively strongly convex]\label{theo:strongconvex}
Suppose that Assumptions~\eqref{ass:intmin}, \myref{ass:variance} and~\eqref{ass:a12} are verified with $\rho > 0$. Let $(M_k)_{k\in \N}$ be a non-increasing positive real-valued sequence such that $M_k > \frac{B_k}{\rho}$, $\forall k \in \N$.  Suppose also that the sequence
$(\bx_k)_{k \in \N}$ is generated by Algorithm~\myref{algo:unified} with $(\alpha_k)_{k \in \N}$ a non-decreasing positive real-valued sequence such that $\alpha_k < \frac{1}{A_k + M_k C_k}$, $\forall k \in \N$. Set $ q_k \coloneqq \max\left\{1 - \alpha_k\gamma_h\mu\left(1-\alpha_k (A_k+ M_k C_k)\right), 1+\frac{B_k}{M_k}-\rho\right\}$. Then for all $k \in \N$, $q_k \in ]0,1[$ and
\begin{equation}
    V_{k+1} \leq q_kV_k + \EE[N_{k}], \nonumber
\end{equation}
where $\displaystyle V_k = \EE\left[\frac{1}{\alpha_{k}^2}  D_{h}(\bxs_*, \bx_{k}) + M_{k}\sigma_{k}^2\right]$, $\forall k \in \N$.
\end{theorem}
\begin{remark}
    In both Theorem~\myref{theo:unifiedmainonlyconvex} and~\myref{theo:strongconvex}, on the right hand sides, the second terms may seem problematic. However, later, for specified variance-reduced algorithms, we will have $N_k \leq 0$ for every $k \in \N$. So they can be dropped. Therefore, in terms of order of convergence, these theorems establish the standard sublinear $\mathcal{O}(1/k)$ and linear $\mathcal{O}(q^k)$ rates for the generic Algorithm~\myref{algo:unified}, corresponding, respectively, to convex and relatively strongly convex functions, without requiring a vanishing stepsize. This is an improvement on BSPPA without variance reduction; see Section~\myref{sect:vanilla}. However, the constants in both theorems depend on the sequence $(G_k)_{k \in \N}$ and this can impede those rates. In the ideal case of Euclidean or quadratic kernel, $G_k = 1$ \citep{pmlr-v139-dragomir21a}, and we recover the generic results of \citet{traore2024variance}. 
\end{remark}
\end{section}
\section{Instantiation of specific algorithms}\label{sect:applications}
In this section, we specialize the generic Algorithm~\myref{algo:unified} with different variance reduction techniques by specifying the term $\be_k$.
\begin{subsection}{Bregman Stochastic Proximal Point Algorithm (BSPPA)}\label{sect:vanilla}
We start by the vanilla BSPPA \citep{davis2018stochastic}, which is also covered by the generic algorithm and analysis by taking $\be_k = 0$ for all $k \in \N$. For BSPPA, the functions can potentially be nonsmooth.
    \begin{algorithm}[BSPPA] \label{algo:sppa}
        Let $(i_k)_{k \in \N}$ be a sequence of i.i.d.~random variables uniformly distributed on $\{1,\dots, n\}$. Let $\bx_0 \equiv \bxs_0 \in \interior\mathcal{C}$ and $\alpha_k>0$ for all $k \in \N$.
        \begin{equation}\label{eq:20230217e}
        \begin{array}{l}
        \text{for}\;k=0,1,\ldots\\
        \left\lfloor
        \begin{array}{l}
        \bx_{k+1} = \argmin_{\bxs \in \HH}\{f_{i_{k}}(\bxs) +\frac{1}{\alpha_k}D_{h}(\bxs,\bx_{k})\}.
        \end{array}
        \right.
        \end{array} 
        \end{equation}
    \end{algorithm}
    \begin{theorem}[$F$ is only convex]\label{theo:bsppa}
        Suppose that Assumptions~\eqref{ass:b1} and~\eqref{ass:b2} hold with $A_k = N_k = 0$ and $B_k \sigma_{k}^2 \leq \sigma^2_{*} \geq 0$ (i.e., $\EE[D_{h}(\bx_{k},\bz_{k+1}) \,\vert\, \Fk_k]$ is bounded by $\alpha_k^2 \sigma^2_{*}$). 
        Suppose also that the sequence $(\bx_k)_{k \in \N}$ is generated by Algorithm~\myref{algo:sppa}. Then, for $k \geq 1$,
        \begin{equation}
            \EE[F(\bbar{\bx}_{k})- F(\bxs_*)] \leq \frac{D_{h}(\bxs_*, \bx_{0})}{\sum_{t=0}^{k-1}\alpha_t} + \sigma^2_{*} \frac{\sum_{t=0}^{k-1}\alpha_t^2}{\sum_{t=0}^{k-1}\alpha_t}, \nonumber
        \end{equation}
        where $\displaystyle \bbar{\bx}_{k} = \sum_{t=0}^{k-1}\frac{\alpha_t}{\sum_{t=0}^{k-1}\alpha_t} \bx_t$.
    \end{theorem}
    \begin{remark}
        Here, $\sigma^2_{*}$ represents a hard bound on $\EE[D_{h}(\bx_{k},\bz_{k+1}) \,\vert\, \Fk_k]$ and, in some sense, encodes the variance of the algorithm. If we set the stepsize to be constant, i.e. $\alpha_k = \alpha$, we obtain
        \begin{align*}
            \EE[F(\bbar{\bx}_{k})- F(\bxs_*)] \leq \frac{D_{h}(\bxs_*, \bx_{0})}{\alpha k} + \alpha \sigma^2_{*}.
        \end{align*}
        This equation shows that the algorithm, because of the variance, will converge to a ball around the minimum rather than the minimum itself and keep oscillating. A vanishing stepsize (typically $(\alpha_k)_{k \in \N}\in \ell_2 \setminus \ell_1$) can be used in order to cancel $\sigma^2_{*}$ asymptotically. But, this slows down the algorithm to a convergence of order $O(1/\sqrt{k})$.
    \end{remark}

    \begin{theorem}\label{theo:bsppa2}
        Suppose that Assumptions~\eqref{ass:a11}, \eqref{ass:b1}, and \eqref{ass:b2} hold with $\beta > 0$ (each $f_i$ is relatively strongly convex), $A_k = N_k = 0$ and $B_{k} \sigma_k^2 \leq \sigma^2_{*} \geq 0$ (i.e., $\EE[D_{h}(\bx_{k},\bz_{k+1}) \,\vert\, \Fk_k]$ is bounded by $\alpha_k^2 \sigma^2_{*}$). Let $\alpha_k = \alpha > 0$, $\forall k \in \N$. Suppose also that the sequence $(\bx_k)_{k \in \N}$ is generated by Algorithm~\myref{algo:sppa}. Then, for $ k \geq 1$,
        \begin{equation}
            \EE[D_{h}(\bxs_*, \bx_{k})]  \leq q^k \EE[D_{h}(\bxs_*, \bx_{0})] +  \alpha^2\frac{1}{1-q} \sigma^2_{*}, \nonumber
        \end{equation}
        where $q = \frac{1}{1+\beta\alpha}$.
    \end{theorem}
\begin{remark}
    As in the only convex case, with a constant stepsize, the algorithm will only converge to a ball around the minimizer and oscillate there due to the variance. Also using a vanishing to kill the variance leads to reduce convergence rate from linear to $O(1/k)$. So we need another way to reduce the variance, that what the next sections will present.
\end{remark}
\begin{remark}[relatively smooth and interpolation case]\label{rm:interpolation} If $f_i$ is relatively smooth for all $i \in [n]$, $B_k\sigma_k^2 \propto \EE_k[D_{h^*}(\nabla h(\bx_{k}) - 2\alpha \nabla f_{i_k}(\bxs_{*}), \nabla h(\bx_{k}))]$; see \citet[Section 3.2]{pmlr-v139-dragomir21a}. Therefore, in the very special case of interpolation, i.e. $\nabla f_i (\bxs_*) = 0$ for all $i \in [n]$, $\sigma^2_{*} = 0$ and BSPPA does converge to the actual minimum and have a good sublinear and linear rates for convex and relatively strongly convex functions, respectively.    
\end{remark}
In the following, we will present the variance-reduced algorithms. We defer the double-loop, SVRG-style variant of the variance-reduced BSPPA to Appendix~\myref{sect:bsvrpsect}, and focus in the main text on its single-loop counterpart. The first variance-reduced BSPPA that we proposed is using a SAGA-style technique, and we coined it BSAPA. It is the Bregman version of SAPA proposed by \citet{traore2024variance}.
\subsection{Bregman SAPA (BSAPA)}
\begin{algorithm}[BSAPA]
\label{algo:sagaprox}
Let $(i_k)_{k \in \N}$ be a sequence of i.i.d.~random variables uniformly distributed on $\{1,\dots, n\}$. Let $\alpha_k>0$ for every $k \in \N$, and set, for every $i \in [n]$, $\bx_0 \equiv \bphi_i^0 \equiv \bxs_0 \in \interior\mathcal{C}$.
\begin{equation}
\begin{array}{l}
\text{for}\;k=0,1,\ldots\\
\left\lfloor
\begin{array}{l}
\begin{array}{rl}
&
\bx_{k+1} = \argmin_{\bxs \in \HH}\Big\{ f_{i_{k}}(\bxs) +\frac{1}{\alpha_k}D_{h}(\bxs,\bx_{k})
\\
& 
-\scalarp{\nabla f_{i_k} (\bphi_{i_k}^k) - \frac{1}{n} \sum_{i=1}^n \nabla f_i (\bphi_i^k),\bxs-\bx_{k}}  \Big\}
\end{array} \\
\forall\, i \in [n]\colon\ \bphi^{k+1}_i = \bphi^k_i + \delta_{i,i_k} (\bx_{k} - \bphi^k_i),
\end{array}
\right.
\end{array} \nonumber
\end{equation}
\end{algorithm}
where $\delta_{i,j}$ is the Kronecker symbol. Here we get $\be_k = \nabla f_{i_k} (\bphi_{i_k}^k) - \frac{1}{n} \sum_{i=1}^n \nabla f_i (\bphi_i^k)$.
We set $\Fk_k = \sigma(i_0, \dots, i_{k-1})$ and $\EE_k[\cdot] = \EE[\cdot\,\vert\, \Fk_k]$. We then have that $\bx_k$ and $\bphi_i^k$ are $\Fk_k$-measurable and $i_k$ is independent of $\Fk_k$. 
Let \(\zeta_k = -2\alpha_k (\nabla f_{i_k}(\bxs_{*}) - \nabla f_{i_k}(\bphi_{i_k}^k)).\) It is clear that $\EE_k[\zeta_k] = 2\alpha_k \frac{1}{n} \sum_{i=1}^n \nabla f_i (\bphi_i^k).$
\begin{lemma}\label{lem:20250311b}
Suppose that Assumptions~\eqref{ass:a13}, \myref{ass:a14} and~\eqref{ass:intmin} hold. Let $s \in \N$ and let $(\bx_k)_{k \in [m]}$ be the sequence generated by the inner iteration in Algorithm~\myref{algo:srvp-l}. We finally assume that there exists a non-increasing sequence $(G_k)_{k \in \N}$ such that, for all $i \in [n]$,
    \begin{align}\label{eq:gain2}
        G_k &\geq G \left(\nabla h(\bx_k), \nabla h(\bx_k), \frac{1}{L_{}} (\nabla f_i(\bx_k) - \nabla f_i(\bxs_{*})) \right), \\
        \label{eq:gain3} G_k &\geq G \Big(\nabla h(\bx_k) - \EE_k[\zeta_k], \nabla h(\bphi_{i}^k), \frac{1}{L_{}} (\nabla f_{i}(\bphi_{i}^k) - \nabla f_{i}(\bxs_{*})) \Big).
    \end{align}
    Then
    \begin{align}\label{eq:varsapa}
        \EE_k[D_h(\bx_{k}, \bz_{k+1})] \leq 2 L \alpha_k^2 G_k D_F(\bx_{k}, \bxs_{*}) + \alpha_k^2 G_k \sigma_k^2 - \frac{1}{2}D_{h^*}(\nabla h(\bx_{k}), \nabla h(\bx_{k}) - \EE_k[\zeta_k]),
    \end{align}
    where 
    \begin{align*}
        \sigma_k^2 = 2 L^2 &\EE_{k} \Big[ D_{h^*} \Big(\nabla h(\bphi_{i_k}^k) - \frac{1}{L} (\nabla f_{i_k}(\bphi_{i_k}^k)\nabla f_{i_k}(\bxs_{*})), \nabla h(\bphi_{i_k}^k) \Big)\Big],
    \end{align*}
    \begin{equation}\label{eq:20221004c}
           \text{and }
           \EE_k\left[\sigma_{k+1}^2\right] \leq \left(1 - \frac{1}{n}\right) \sigma_k^2  + \frac{2L}{n}D_F(\bx_{k}, \bxs_{*}).
    \end{equation}
\end{lemma}
\begin{remark}
    As we stated earlier in Remark~\myref{rk:assumproof}, Lemma~\myref{lem:20250311b} shows that Assumptions~\myref{ass:variance} are verified with $A_k = 2LG_k$, $B_k = G_k$, $ N_k = - (1/2\alpha_k^2)D_{h^*}(\nabla h(\bx_{k}), \nabla h(\bx_{k}) - \EE_k[\zeta_k]) $, $\rho = 1/n$ and $C_k = C = \frac{2L}{n}$. By just putting this values in Theorems~\myref{theo:unifiedmainonlyconvex} and~\myref{theo:strongconvex} we obtain the following two corollaries, respectively.
\end{remark}
\begin{assumption}\label{ass:gk}
For all \(k \geq 1\), we assume \(G_k\) is deterministic and not a random variable. If necessary, we may take \(G_k = G_0, \forall k \in \N\). For upper bounds on the gain function \(G\), see \citep{pmlr-v139-dragomir21a, hanzely2021accelerated}.
\end{assumption}
\begin{corollary}[$F$ is only convex]\label{cor:sapaconv} Let the assumptions of Lemma~\myref{lem:20250311b} and Assumption~\ref{ass:gk} hold.
 Let $(M_k)_{k\in \N}$ be a non-increasing positive real-valued sequence such that $M_k \geq nG_k$, $\forall k \in \N$. Suppose also that the sequence $(\bx_k)_{k \in \N}$ is generated by Algorithm~\myref{algo:srvp-l} with $(\alpha_k)_{k \in \N}$ a non-decreasing positive real-valued sequence such that $\alpha_k < \frac{1}{2L(G_k + M_k/n)}$, $\forall k \in \N$. Then, for all $k \in \N$,
\begin{align}
    \EE[F(\bbar{\bx}_{k})- F(\bxs_*)] &\leq \frac{(1/\alpha_{0}^2) \EE[ D_{h}(\bxs_*, \bx_{0})] + M_{0}\EE[\sigma_{0}^2]}{\sum_{t=0}^{k-1} (1/\alpha_{t})\left(1-2\alpha_tL(G_t + M_t/n)\right)} 
     \nonumber,
\end{align}
with
\begin{equation*}
    \displaystyle \bbar{\bx}_k = \sum_{t=0}^{k-1} \frac{(1/\alpha_{t})\left(1-2\alpha_tL(G_t + M_t/n)\right)}{\sum_{t=0}^{k-1} (1/\alpha_{t})\left(1-2\alpha_tL(G_t + M_t/n)\right)} \bx_t.
\end{equation*}
\end{corollary}
\begin{corollary}[$F$ is $\mu$-relatively strongly convex]\label{cor:sapatrongconvex} Let the assumptions of Lemma~\myref{lem:20250311b} and Assumption~\ref{ass:gk} hold.
Let $(M_k)_{k\in \N}$ be a non-increasing positive real-valued sequence such that $M_k > nG_k$, $\forall k \in \N$.  Suppose also that the sequence
$(\bx_k)_{k \in \N}$ is generated by Algorithm~\myref{algo:srvp-l} with $(\alpha_k)_{k \in \N}$ a non-decreasing positive real-valued sequence such that $\alpha_k < \frac{1}{2L(G_k + M_k/n)}$, $\forall k \in \N$. Set $ q_k \coloneqq \max\left\{1 - \alpha_k\gamma_h\mu\left(1-2\alpha_kL(G_k + M_k/n)\right), 1+\frac{G_k}{M_k}-\frac{1}{n}\right\}$. Then, for all $k \in \N$, $q_k \in ]0,1[$,
\begin{equation}
    V_{k+1} \leq q_kV_k\;\; \text{and}\;\; V_{k+1} \leq \left(\prod_{t=0}^{k}q_t\right)V_0,
\end{equation}
where
\begin{equation*}
    \displaystyle V_k = \EE\left[\frac{1}{\alpha_{k}^2}  D_{h}(\bxs_*, \bx_{k}) + M_{k}\sigma_{k}^2\right].
\end{equation*}
Using relative smoothness, we get
\begin{align*}
    D_{F}(\bxs_*, \bx_{k+1}) \leq \left(\prod_{t=0}^{k}q_t\right)\alpha_{k+1}^2LV_0.
\end{align*}
\end{corollary}
\begin{remark} To see a normal linear rate, we can set $\alpha_k = \alpha = \frac{1}{2L(G_0 + M_0/n)+1} \leq \frac{1}{2L(G_k + M_k/n)}$, for all $k \in \N$. We can always take $M_k$ such that $\frac{G_k}{M_k} = \frac{1}{n+1}$. Then $(q_k)_{k \in \N}$ is non-increasing,
\begin{equation*}
    V_{k+1} \leq q_0 V_k\;\; \text{and}\;\; V_{k+1} \leq q_0^{k+1} V_0.
\end{equation*}
Also 
\begin{align*}
    D_{F}(\bxs_*, \bx_{k+1}) \leq q_0^{k+1} L \EE\left[D_{h}(\bxs_*, \bx_{0}) + \alpha^2M_{0}\sigma_{0}^2\right].
\end{align*}
\end{remark}
\subsection{Bregman Loopless SVRP (BLSVRP)}
 We also proposed BLSVRP, using an L-SVRG-style technique with only one loop; see Algorithm \myref{algo:srvp-l}. The second loop is replaced by a Bernoulli probability. It generalizes, to the Bregman distance, L-SVRP found for instance in \citet{khaled2022faster} and \citet{traore2024variance}.
\begin{algorithm}[BLSVRP]
 \label{algo:srvp-l}
Let $(i_k)_{k \in \N}$ be a sequence of i.i.d.~random variables uniformly distributed on $\{1,\dots, n\}$ and let $(\varepsilon_k)_{k \in \N}$ be a sequence of i.i.d~Bernoulli random variables such that $\mathsf{P}(\varepsilon_k=1)=p \in \left(0,1\right]$. Let $\alpha_k > 0$ for every $k \in \N$, and
set $\bx_0 \equiv \bu_0 \equiv \bxs_0 \in \interior\mathcal{C}$.
\begin{equation}
\begin{array}{l}
\text{for}\;k=0,1,\ldots\\
\left\lfloor
\begin{array}{l}
\begin{array}{rl}
&
\bx_{k+1} = \argmin_{\bxs \in \HH}\Big\{f_{i_{k}}(\bxs) - \scalarp{\nabla f_{i_{k}}(\bu_{k}) - \nabla F(\bu_{k}),\bxs-\bx_{k}} +\frac{1}{\alpha_k}D_{h}(\bxs,\bx_{k}) \Big\}
\end{array} \\
\bu_{k+1}=(1-\varepsilon_k)\bu_k + \varepsilon_k \bx_k
\end{array}
\right.
\end{array} \nonumber
\end{equation}
\vspace{-0.75em}
\end{algorithm}
For BLSVRP, $\be_k = \nabla f_{i_{k}}(\bu_{k}) - \nabla F(\bu_{k})$. Set
 $\Fk_k = \sigma(i_0, \dots, i_{k-1}, \varepsilon_0, \dots, \varepsilon^{k-1})$ and $\EE_k[\cdot] = \EE[\cdot\,\vert\, \Fk_k]$.
We then have that $\bx_k$, $\bu_k$ and $\by_k$ are $\Fk_k$-measurable, $i_k$ and $\varepsilon_k$ are independent of $\Fk_k$. Let $\zeta_k = -2\alpha_k (\nabla f_{i_k}(\bxs_{*}) - \nabla f_{i_k}(\bu_{k}))$. So, $\EE_k[\zeta_k] = 2\alpha_k \nabla F(\bu_{k})$.
\begin{lemma}\label{lem:20250311a}
Suppose that Assumptions~\eqref{ass:a13}, \myref{ass:a14} and~\eqref{ass:intmin} hold. Let $s \in \N$ and let $(\bx_k)_{k \in [m]}$ be the sequence generated by the inner iteration in Algorithm~\myref{algo:srvp-l}. Assume that there exists a non-increasing sequence $(G_k)_{k \in \N}$ such that, for all $i \in [n]$,
    \begin{align}\label{eq:gain4}
        G_k &\geq G \left(\nabla h(\bx_k), \nabla h(\bx_k), \frac{1}{L_{}} (\nabla f_i(\bx_k) - \nabla f_i(\bxs_{*})) \right),  \\
        \label{eq:gain5} G_k &\geq G \Big(\nabla h(\bx_k) - \EE_k[\zeta_k], \nabla h(\bu_{k}), \frac{1}{L_{}} (\nabla f_i(\bu_{k}) - \nabla f_i(\bxs_{*})) \Big) .
    \end{align}
    Then
    \begin{align}\label{eq:varlsvrp}
        &\EE_k[D_h(\bx_{k}, \bz_{k+1})] 
        \leq 2 L \alpha_k^2 G_k D_F(\bx_{k}, \bxs_{*}) + \alpha_k^2 G_k \sigma_k^2 - \frac{1}{2} D_{h^*}(\nabla h(\bx_{k}), \nabla h(\bx_{k}) - \EE_k[\zeta_k]),
    \end{align}
    where 
    \begin{align*}
        \sigma_k^2 = 2 L^2 &\EE_{k} \Big[ D_{h^*} \Big(\nabla h(\bu_{k}) - \frac{1}{L} (\nabla f_{i_k}(\bu_{k}) - \nabla f_{i_k}(\bxs_{*})), \nabla h(\bu_{k}) \Big)\Big],
    \end{align*}
    \begin{equation}\label{eq:20221103b}
        \text{and }\;\; 
        \EE_k\big[\sigma_{k+1}^2\big] \leq (1-p)\sigma_k^2 + 2 p L D_F(\bx_{k}, \bxs_{*}).
    \end{equation}
\end{lemma}
\begin{remark}
    Lemma~\myref{lem:20250311a} shows that Assumptions~\myref{ass:variance} are verified with $A_k = 2LG_k$, $B_k = G_k$, $ N_k = - (1/2\alpha_k^2) D_{h^*}(\nabla h(\bx_{k}), \nabla h(\bx_{k}) - \EE_k[\zeta_k]) $, $\rho = p$ and $C_k = C = 2pL$.
\end{remark}
\begin{corollary}[$F$ is only convex]\label{cor:lsvrpconv} Let the assumptions of Lemma~\myref{lem:20250311a} and Assumption~\ref{ass:gk} hold.
 Let $(M_k)_{k\in \N}$ be a non-increasing positive real-valued sequence such that $M_k \geq \frac{G_k}{p}$, $\forall k \in \N$. Suppose also that the sequence $(\bx_k)_{k \in \N}$ is generated by Algorithm~\myref{algo:srvp-l} with $(\alpha_k)_{k \in \N}$ a non-decreasing positive real-valued sequence such that $\alpha_k < \frac{1}{2L(G_k + M_k p)}$, $\forall k \in \N$. Then, for all $k \in \N$,
\begin{align}
    \EE[F(\bbar{\bx}_{k})- F(\bxs_*)] &\leq \frac{(1/\alpha_{0}^2) \EE[ D_{h}(\bxs_*, \bx_{0})] + M_{0}\EE[\sigma_{0}^2]}{\sum_{t=0}^{k-1} (1/\alpha_{t})\left(1-2\alpha_tL(G_t + M_t p)\right)} 
     \nonumber,
\end{align}
with
\begin{equation*}
    \displaystyle \bbar{\bx}_k = \sum_{t=0}^{k-1} \frac{(1/\alpha_{t})\left(1-2\alpha_tL(G_t + M_t p)\right)}{\sum_{t=0}^{k-1} (1/\alpha_{t})\left(1-2\alpha_tL(G_t + M_t p)\right)} \bx_t.
\end{equation*}
\end{corollary}
\begin{corollary}[$F$ is $\mu$-relatively strongly convex]\label{cor:lsvrpstrongconvex} Let the assumptions of Lemma~\myref{lem:20250311a} and Assumption~\ref{ass:gk} hold.
Let $(M_k)_{k\in \N}$ be a non-increasing positive real-valued sequence such that $M_k > \frac{G_k}{p}$, $\forall k \in \N$.  Suppose also that the sequence
$(\bx_k)_{k \in \N}$ is generated by Algorithm~\myref{algo:srvp-l} with $(\alpha_k)_{k \in \N}$ a non-decreasing positive real-valued sequence such that $\alpha_k < \frac{1}{2L(G_k + M_k p)}$, $\forall k \in \N$. Set $ q_k \coloneqq \max\left\{1 - \alpha_k\gamma_h\mu\left(1-2\alpha_kL(G_k + M_k p)\right), 1+\frac{G_k}{M_k}-p\right\}$. Then, for all $k \in \N$, $q_k \in ]0,1[$,
\begin{equation}
    V_{k+1} \leq q_kV_k\;\; \text{and}\;\; V_{k+1} \leq \left(\prod_{t=0}^{k}q_t\right)V_0,
\end{equation}
where
\begin{equation*}
    \displaystyle V_k = \EE\left[\frac{1}{\alpha_{k}^2}  D_{h}(\bxs_*, \bx_{k}) + M_{k}\sigma_{k}^2\right].
\end{equation*}
Using relative smoothness, we get
\begin{align*}
    D_{F}(\bxs_*, \bx_{k+1}) \leq \left(\prod_{t=0}^{k}q_t\right)\alpha_{k+1}^2LV_0.
\end{align*}
\end{corollary}
\end{subsection}

\section{Experiments}\label{sect:experiments} 
\begin{figure*}[ht]
     \begin{subfigure}[b]{0.22\textwidth}
         \centering
         \includegraphics[width=\textwidth]
         {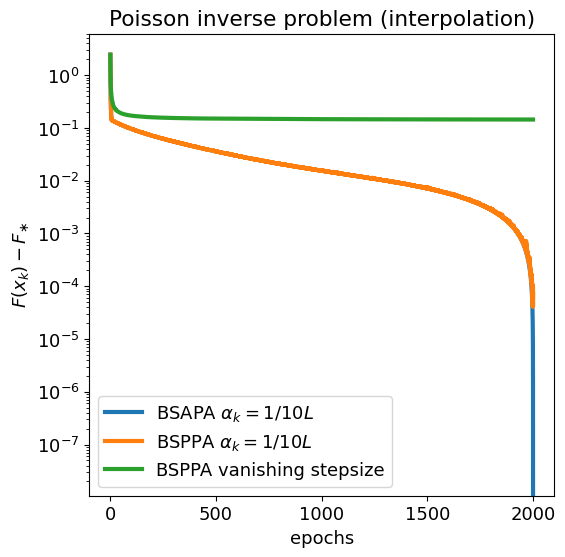}
     \end{subfigure}
     \begin{subfigure}[b]{0.22\textwidth}
         \centering
         \includegraphics[width=\textwidth]
         {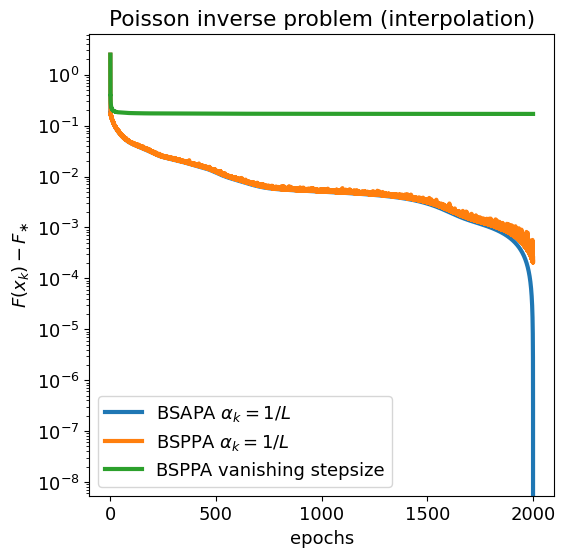}
     \end{subfigure}
     \begin{subfigure}[b]{0.23\textwidth}
         \centering
         \includegraphics[width=\textwidth]
         {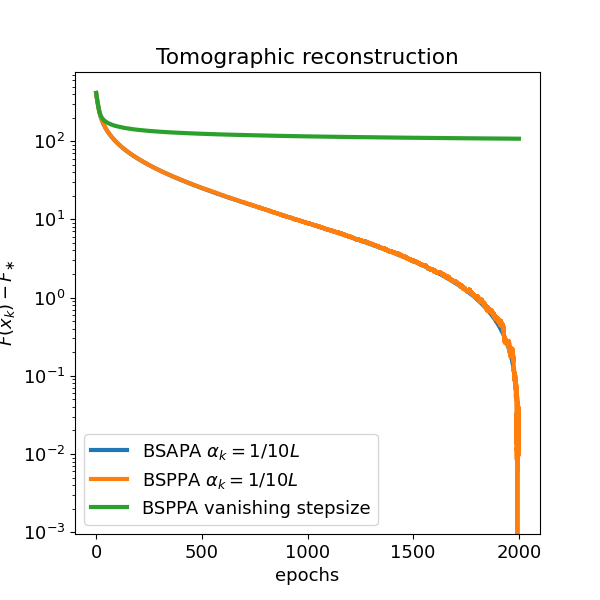}
     \end{subfigure}
     \begin{subfigure}[b]{0.23\textwidth}
         \centering
         \includegraphics[width=\textwidth]
         {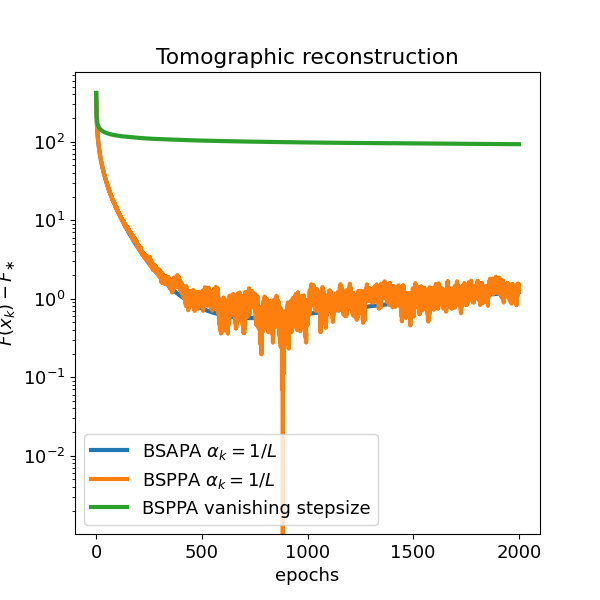}
     \end{subfigure}
    \caption{
    Our BSAPA (variance reduced) is more stable, converges to the minimum and does not oscillate around it, especially in the non-interpolation case with constant stepsize, contrary to BSPPA.
    }
    \label{fig:tomo}
\end{figure*}
    The following experiments go beyond the theory in certain aspects, which are pointed out. However, they still demonstrate the advantages of variance reduction. Using BSAPA, we illustrate those advantages by performing numerical experiments on the Poisson linear inverse problem given by:
    \begin{equation*}
        \min_{\bxs \in \mathbb{R}_{+}^d} F(\bxs) = \frac{1}{n}D_{\text{KL}}(b, A\bxs),
    \end{equation*}
    where
    \begin{align*}
    D_{\text{KL}}(b, A\bxs) \coloneqq \sum_{i=1}^n \Big\{f_i(\bxs) &\coloneqq \mathsf{b}_i \log(\mathsf{b}_i/(A\bxs)_i)  - \mathsf{b}_i + (A\bxs)_i \Big\}
    \end{align*}
    is the Kullback-Leibler divergence, $A \in \mathbb{R}^{n \times d}_{+}$ is the forward operator of the inverse problem, and $b \in \mathbb{R}_{++}^n$ is the measurements vector. It models the maximum likelihood estimation problem where the model is $b \sim \text{Poisson}(A\bxs_*)$, with $\bxs_*$ the true unknown value. Each $f_i$ is convex. The kernel used is $h(\bxs) = -\sum_{i=0}^d \log \xx_i$. Based on \citet{bauschke2017descent}, we have that each function $f_i$ is $L$-relatively smooth w.r.t. $h$ when $L= \max_i b_i$. In addition, \(\dom h^* = (-\infty, 0)\) which goes beyond the theoretical analysis. In these experiments, we compare a fixed stepsize BSAPA to non variance-reduced BSPPA with both fixed and vanishing stepsizes. 
    
    At each iteration, we solved the proximal subproblem
    \begin{equation*}
        \argmin_{\bxs \in \mathbb{R}_{+}^d} \left\{ f_{i_{k}}(\bxs) 
    - \langle \be_k, \bxs - \bx_{k}\rangle 
    + \tfrac{1}{\alpha_k} D_{h}(\bxs,\bx_{k}) \right\}
    \end{equation*}
    using a gradient descent subroutine, since no closed-form solution is available in general, except when $A$ is diagonal. As this subroutine is employed in both algorithms under comparison, the evaluation remains equitable. 
    Strictly speaking, this procedure corresponds to an approximate minimization of the proximal mapping, and hence the implementation is closer to an inexact variant of the algorithms. 
    A formal treatment of inexactness lies beyond the scope of this work. 
    Nevertheless, this approximation does not affect our empirical findings, which still clearly demonstrate the advantage of variance reduction. These findings encourage future extensions of the theoretical analysis toward inexact variants.

    For the first experiment, we used a uniformly generated matrix. We set $n=500, d = 100$. 
    We couldn't choose the sequence \((G_k)_{k \in \N}\), as the gain function for the kernel is unknown exactly. Following previous work \citep{pmlr-v139-dragomir21a} and their numerical experiments, we therefore assume \((G_k)_{k \in \N}\) is constant. To account for this uncertainty and also the fact that $L$ is quite conservative, we consider several values for the stepsize in our experiments. 
    Here, we present the results for two values; 
    see Figure~\myref{fig:tomo}. 
    For the complete range of stepsizes, see Appendix~\myref{app:figures}, where all additional experimental results are provided.
    As expected, the results show that with vanishing stepsize (green curve), SPPA is slower that BSAPA (blue curve), and even plateaus quite early. With constant stepsize (orange curve), SPPA has the same rate as BSAPA (interpolation). However, while the stepsizes increase we observe that SPPA is less stable than BSAPA. More stability can be added to SPPA by using average iterates $(\bbar{\bx})_{k \in \N}$. This instability of SPPA, in the interpolation case, is vastly due to the approximate proximal mapping. Indeed when $A$ is diagonal, a closed form exists for the proximal mapping and the experiments don't exhibit this instability for SPPA; see Figure~\myref{fig:tomodiag}.
    Nevertheless, the experiments shows the advantage and stability of variance reduction, even in the interpolation case, when the proximal mapping can only be approximated.
    
    To cover the non-interpolation case, we did some experiments on tomographic reconstruction problem, where $A$ is a discrete Radon transform, that projects the image $\bxs$ in $n$ different angles $(\theta_1, \ldots, \theta_n)$, with $n = 90$. The optimization problem is 
    \begin{align*}
      \min_{\bxs \in \R^d_+}\frac{1}{n}\sum_{i=1}^n \Big\{f_i(\bxs) \coloneqq &\mathsf{b}_i \log(\mathsf{b}_i/(A\bxs)_{\theta_i}) - \mathsf{b}_i + (A\bxs)_{\theta_i}\Big\}.
    \end{align*}
    The image used is the Shepp-Logan phantom. Also here, the vanishing stepsize SPPA is slower than BSAPA; see Figure~\myref{fig:tomo}. 
    Furthermore, the results show, that constant stepsize SPPA oscillates around the minimum while BSAPA is stable; expected behaviors for non-interpolation cases. These oscillations are low for small stepsizes and high for bigger ones. 
    We also ran this for different stepsizes because $G_k$ is unknown and $L$ is conservative. 

    Finally, in all experiments, BSAPA tends to explode quicker than SPPA when the stepsize is big enough.
    This is seen from the theory where SPPA does not have a bound on the stepsize; see Theorems~\myref{theo:bsppa} and~\myref{theo:bsppa2}. However bigger stepsizes increase the term which is related to the variance in those theorems, hence an increase in oscillations and no convergence. But until that stepsize threshold imposed by the theory, for non-interpolation cases, BSAPA remains more stable and ensure convergence whereas SPPA does not.

\section{Conclusion}
    In this paper, we conducted a unified analysis of variance reduction for the Bregman stochastic proximal point algorithm (BSPPA). We proposed a generic variance-reduced algorithm based on BSPPA and prove convergence rates. From that general algorithm and analysis, we derived several new variance-reduced BSPPAs, employing SVRG- and SAGA-like techniques, and their corresponding convergence rates. More specifically, under the relative smoothness assumption, we prove sublinear and linear rates for convex and relatively strongly convex functions, respectively. Our general analysis can also recover the previously studied vanilla BSPPA with its standard rates.  Furthermore, the unified theoretical results extend seamlessly to variance-reduced SGD with Bregman distance. Although not fully captured by the theory, the experiments demonstrated the advantages of variance reduction and highlighted the need for further refinement of the theoretical analysis. For future work, we will focus on extending this work to nonsmooth functions and/or to inexact proximal mapping variants.

\runinsubsectionstar{Acknowledgements}
P. Ochs acknowledges funding by the German Research Foundation for the project DFG Grant OC 150/3-1. The majority of this work was done while C. Traor\'e was a postdoc at Saarland University. He acknowledges the support of Occitanie region, the European Regional Development Fund (ERDF), and the French government, through the France 2030 project managed by the National Research Agency (ANR) with the reference number ``ANR-22-EXES-0015''.

\appendix
\section{Bregman SVRP (BSVRP)}\label{sect:bsvrpsect}
Our analysis can also be adapted to the original SVRG-like variance reduction, i.e. a Bregman version of SVRP in \citet{traore2024variance}. We call it BSVRP.
\begin{algorithm}[BSVRP]
\label{algo:srvprox2}
Let $m\in \N$, with $m\geq 1$, and $(\xi_s)_{s \in \N}$, $(i_t)_{t \in \N}$ be two independent sequences of i.i.d.~random variables uniformly distributed on $\{0,1,\dots, m-1\}$ and $\{1,\dots, n\}$ respectively. 
Let $\alpha_k > 0$ for every $k \in \N$, and set $\tilde{\bx}_0 \equiv \tilde{\bxs}_0\in \interior\mathcal{C}$.
\begin{equation}
\begin{array}{l}
\text{for}\;s=0,1,\ldots\\
\left\lfloor
\begin{array}{l}
\bx_0 = \tilde{\bx}_{s} \\
\text{for}\;k=0,\dots, m-1\\[0.7ex]
\left\lfloor
\begin{array}{rl}
\bx_{k+1} = \argmin_{\bxs \in \HH}\Big\{f_{i_{k}}(\bxs) 
-\scalarp{\nabla f_{i_{sm+k}}(\tilde{\bx}_{s}) - \nabla F(\tilde{\bx}_{s}),\bxs-\bx_{k}} +\frac{1}{\alpha_k}D_{h}(\bxs,\bx_{k}) \Big\} \\
\end{array}
\right. \\
\tilde{\bx}_{s+1} = \sum_{k=0}^{m-1} \delta_{k,\xi_s}\bx_{k}, 
\text{ or }
{\tilde{\bx}_{s+1} = \frac{1}{m} \sum_{k=0}^{m-1} \bx_{k},}
\end{array}
\right.
\end{array} \nonumber
\end{equation}
\end{algorithm}
where $\delta_{k,h}$ is the Kronecker symbol. It follows that $\be_k = \nabla f_{i_{sm+k}}(\tilde{\bx}_{s}) - \nabla F(\tilde{\bx}_{s})$. Moreover, setting
 $\Fk_{s,k} = \sigma(\xi_0, \dots, \xi_{s-1, }i_0, \dots, i_{sm+k-1}),$
we have that $\tilde{\bx}_{s}, \bxs_*$,
and $\bx_k$ are $\Fk_{s,k}$-measurable and $i_{sm + k}$ is independent of $\Fk_{s,k}$. 
Let $\zeta_k = -2\alpha_k (\nabla f_{i_k}(\bxs_{*}) - \nabla f_{i_k}(\tilde{\bx}_{s}))$. That means that $\EE[\zeta_k\,\vert\, \Fk_{s,k}] = 2\alpha_k \nabla F(\tilde{\bx}_{s})$.
\begin{lemma}\label{lem:20220927a}
Suppose that Assumptions~\eqref{ass:a13}, \myref{ass:a14} and~\eqref{ass:intmin} hold. We also assume that there exists a non-increasing sequence $(G_{\ell})_{\ell \in \N}$ such that, for all $i \in [n]$,
 \begin{align}\label{gain6}
        G_{sm + k} &\geq G \left(\nabla h(\bx_k), \nabla h(\bx_k), \frac{1}{L_{}} (\nabla f_i(\bx_k) - \nabla f_i(\bxs_*)) \right),  \\
        \label{gain7} G_{sm+k} &\geq G \Big(\nabla h(\bx_k) - \EE[\zeta_k\,\vert\, \Fk_{s,k}], \nabla h(\tilde{\bx}_{s}),
        \frac{1}{L_{}} (\nabla f_i(\tilde{\bx}_{s}) - \nabla f_i(\bxs_*)) \Big),
    \end{align}
for any $s \in \N$, \(k \in \{0,1,\cdots, m-1\}\) and $(\bx_k)_{k \in [m]}$ generated by the inner iteration of Algorithm~\myref{algo:srvprox2}. 

So let $s \in \N$ and let $(\bx_k)_{k \in [m]}$ be the sequence generated by the inner iteration in Algorithm~\myref{algo:srvprox2}. Then, for all $k \in \{0,1,\cdots, m-1\}$,
    \begin{align*}
        \EE[D_h(\bx_{k}, \bz_{k+1})\,\vert\, \Fk_{s,k}] 
        &\leq 2 L \alpha_k^2 G_{sm + k} D_F(\bx_{k}, \bxs_{*}) +  \alpha_k^2 G_{sm + k} \sigma_k^2 
        - \frac{1}{2}D_{h^*}\left(\nabla h(\bx_{k}), \nabla h(\bx_{k}) + 2\alpha_k\nabla F(\tilde{\bx}_{s})\right) \nonumber, 
    \end{align*}
    where 
    \begin{align*}
      \sigma_k^2 = 2 L^2 &\EE \Bigg[ D_{h^*} \Big(\nabla h(\tilde{\bx}_{s}) - \frac{1}{L} (\nabla f_{i_k}(\tilde{\bx}_{s})
        - \nabla f_{i_k}(\bxs_{*})), \nabla h(\tilde{\bx}_{s}) \Big)\,\vert\, \Fk_{s,k}\Bigg],
    \end{align*}
    and, trivially, $\EE\big[\sigma_{k+1}^2 \,\vert\,\Fk_{s,k}\big] = \EE\big[\sigma_k^2 \,\vert\, \Fk_{s,k}\big].$
\end{lemma}
\begin{remark}\label{rmk:constsvrg}
    By Lemma~\myref{lem:20220927a}, Assumptions~\myref{ass:variance} are satisfied with $A_k = 2LG_{sm + k}$, $B_k = G_{sm + k}$, $N_k = - (1/2\alpha_k^2) D_{h^*}\left(\nabla h(\bx_{k}), \nabla h(\bx_{k}) + 2\alpha_k\nabla F(\tilde{\bx}_{s})\right)$, $\rho = C_k = 0$. For simplicity, we set $A_k = A = 2LG_{sm + k} = 2LG_0$, $B_k = B = G_{sm + k} = G_0$, $\alpha_k = \alpha$.
\end{remark}
\begin{theorem}[$F$ is $\mu$-relatively strongly convex]\label{theo:bsvrp} We assume~\eqref{ass:a12}  and that 
$\alpha < \frac{1}{4LG_0}$.
Then, for all $s \in \N$ and under the assumptions of Lemma~\myref{lem:20220927a},
    \begin{align*}
        &\EE[D_F(\tilde{\bx}_{s+1}, \bxs_{*})]
        \leq \Bigg(\frac{1}{\gamma_h\mu\alpha\left(1-2L\alpha G_0\right)m} + \frac{2L\alpha G_0}{1-2L\alpha G_0}\Bigg)\EE[D_{F}(\tilde{\bx}_{s}, \bxs_*)]. 
    \end{align*}
\end{theorem}
\begin{remark}
    Taking $\displaystyle\alpha < \frac{1}{4LG_0}$
    ensures that $\displaystyle \frac{2L\alpha G_0}{1-2L\alpha G_0} < 1$. Then taking $m$ big enough gives linear rate.
\end{remark}
\section{{Proofs}}\label{app:proof}
\subsection{Fundamental results}
For the proofs, we start with a proposition that constitutes the cornerstone of our analysis. Most of the others results are derived from that proposition.
\begin{proposition}\label{prop:unified}
Suppose that Assumptions~\myref{ass:variance} are verified, Assumption~\eqref{ass:a11} holds, and that the sequence
$(\bx_k)_{k \in \N}$ is generated by Algorithm~\myref{algo:unified}. Then, for all $k \in \N$,
\begin{align}
    \left(1+\beta\alpha_k\right) \EE[D_{h}(\bxs_*, \bx_{k+1})\,\vert\, \Fk_k]  
    &\leq D_{h}(\bxs_*, \bx_{k}) -\alpha_k\left[1-\alpha_k A_k\right]\left(F(\bx_{k})- F(\bxs_*)\right) 
     + \alpha_k^2 B_k \sigma_k^2  + \alpha_k^2 N_{k} \nonumber.
\end{align}
\end{proposition}
\begin{proof}
    We recall the definition of $\bx_{k+1}$:
\begin{align}
    \bx_{k+1} &=\argmin_{\bxs \in \HH}\Big\{\underbrace{f_{i_{k}}(\bxs)-\scalarp{\be_k,\bxs-\bx_{k}} - \beta D_h(\bxs, \bx_k)}_{=:R_{i_{k}}(\bxs)}  + \left(\frac{1+\beta\alpha_k}{\alpha_k}\right)D_{h}(\bxs,\bx_{k})\Big\} \nonumber \\
    &= \label{eq:20250219c} 
    \nabla h^* \left(\nabla h (\bx_{k}) - \left(\frac{1+\beta\alpha_k}{\alpha_k}\right)^{-1} \br_{k+1} \right), 
\end{align}
 where \(\br_{k+1} \in \partial R_{i_k}(\bx_{k+1}) \) such that~\eqref{eq:20250219c} holds. \\
Let $\bxs \in \mathcal{C}$. The function \(R_{i_{k}}(\bxs)=f_{i_{k}}(\bxs)-\scalarp{\be_k,\bxs-\bx_{k}}- \beta D_h(\bxs, \bx_k)\) is convex, since $f_{i_k}$ is $\beta$-relatively strongly convex w.r.t. $h$.
By convexity of $R_{i_{k}}$, we have:
\begin{align}\label{eq:20220915a}
    \scalarp{\br_{k+1} , \bxs-\bx_{k+1}} &\leq R_{i_{k}}(\bxs)-R_{i_{k}}(\bx_{k+1}) \nonumber \\
    &=f_{i_{k}}(\bxs)-f_{i_{k}}(\bx_{k+1})-\scalarp{\be_k , \bxs-\bx_{k}} + \scalarp{\be_k , \bx_{k+1}-\bx_{k}} \nonumber \\
    &\qquad - \beta D_h(\bxs, \bx_k) + \beta D_h(\bx_{k+1}, \bx_{k}).
\end{align}
Since \(\br_{k+1} =\left(\frac{1+\beta\alpha_k}{\alpha_k}\right)\left(\nabla h(\bx_{k})-\nabla h(\bx_{k+1})\right)\) by~\eqref{eq:20250219c}, it follows from~\eqref{eq:20220915a}:
\begin{align*}
  &\left(\frac{1+\beta\alpha_k}{\alpha_k}\right)\scalarp{\nabla h(\bx_{k})-\nabla h(\bx_{k+1}) , \bxs-\bx_{k+1}} \\
  &\leq f_{i_{k}}(\bxs)-f_{i_{k}}(\bx_{k+1})-\scalarp{\be_k , \bxs-\bx_{k}} + \scalarp{\be_k , \bx_{k+1}-\bx_{k}} \\
  &\qquad - \beta D_h(\bxs, \bx_k) + \beta D_h(\bx_{k+1}, \bx_{k}).
\end{align*}
By using the three points identity in Lemma~\myref{lem:20250424a},
\(\scalarp{\nabla h(\bx_{k})-\nabla h(\bx_{k+1}) , \bxs-\bx_{k+1}}=D_{h}(\bx_{k+1},\bx_{k}) + D_{h}(\bxs, \bx_{k+1}) - D_{h}(\bxs, \bx_{k})\),
in the previous inequality, we find
\begin{align}
\label{eq:20220915f}
-\scalarp{\be_k, \bx_{k+1}-\bx_{k}} + f_{i_{k}}(\bx_{k+1})- f_{i_{k}}(\bxs) + \frac{1}{\alpha_k}D_{h}(\bx_{k+1},\bx_{k}) &\leq
         \frac{1}{\alpha_k} D_{h}(\bxs, \bx_{k}) - \frac{\left(1+\beta\alpha_k\right)}{\alpha_k} D_{h}(\bxs, \bx_{k+1}) \nonumber \\
         &\qquad - \scalarp{\be_k , \bxs-\bx_{k}}.
\end{align}
We recall that the explicit direction $\bv_k$ can be rewritten as 
\begin{align}\label{eq:10022025b}
  \bv_k = \frac{1}{\alpha_k} \left(\nabla h(\bx_k) - \nabla h(\bz_{k+1})\right).
\end{align} 
Using \eqref{eq:10022025c}, we can lower bound the left hand side as follows:
\begin{align}
-\scalarp{\be_k , \bx_{k+1}-\bx_{k}} + f_{i_{k}}(\bx_{k+1})- f_{i_{k}}(\bxs) + \frac{1}{\alpha_k}D_{h}(\bx_{k+1},\bx_{k}) &= -\scalarp{\be_k , \bx_{k+1}-\bx_{k}} + f_{i_{k}}(\bx_{k+1})-f_{i_{k}}(\bx_{k}) \nonumber \\
&\qquad + \frac{1}{\alpha_k}D_{h}(\bx_{k+1},\bx_{k}) + f_{i_{k}}(\bx_{k})- f_{i_{k}}(\bxs) \nonumber\\
& \geq -\scalarp{\be_k , \bx_{k+1}-\bx_{k}} + \scalarp{\bg_{k} ,\bx_{k+1}-\bx_{k}} \nonumber \\
&\qquad + \frac{1}{\alpha_k}D_{h}(\bx_{k+1},\bx_{k}) +f_{i_{k}}(\bx_{k})- f_{i_{k}}(\bxs) \nonumber \\
&= \scalarp{-\be_k + \bg_{k}, \bx_{k+1}-\bx_{k}} +  \frac{1}{\alpha_k}D_{h}(\bx_{k+1},\bx_{k}) \nonumber \\
&\qquad +f_{i_{k}}(\bx_{k})- f_{i_{k}}(\bxs) \nonumber \\
&= \scalarp{\bv_k , \bx_{k+1}-\bx_{k}} +  \frac{1}{\alpha_k}D_{h}(\bx_{k+1},\bx_{k}) \nonumber \\
&\qquad + f_{i_{k}}(\bx_{k})- f_{i_{k}}(\bxs) \nonumber,
\end{align}
where in the first inequality, we used convexity of $f_i$, for all $i \in [n]$. Using the definitions of $\bv_k$ in~\eqref{eq:10022025b}, we can write
\begin{align}
-\scalarp{\be_k , \bx_{k+1}-\bx_{k}} + f_{i_{k}}(\bx_{k+1})- f_{i_{k}}(\bxs) + \frac{1}{\alpha_k}D_{h}(\bx_{k+1},\bx_{k})
&=  \frac{1}{\alpha_k} \scalarp{ \nabla h(\bx_k) - \nabla h(\bz_{k+1}) , \bx_{k+1}-\bx_{k}} \nonumber \\
&\qquad +  \frac{1}{\alpha_k}D_{h}(\bx_{k+1},\bx_{k}) + f_{i_{k}}(\bx_{k})- f_{i_{k}}(\bxs) \nonumber \\
&\label{eq:20220915d}= -\frac{1}{\alpha_k}D_{h}(\bx_{k},\bz_{k+1}) +  \frac{1}{\alpha_k}D_{h}(\bx_{k+1},\bz_{k+1}) \nonumber\\
&\qquad + f_{i_{k}}(\bx_{k})- f_{i_{k}}(\bxs),
\end{align}
where the last inequality uses the three points inequality~\eqref{lem:20250424a}. 
By using~\eqref{eq:20220915d} in~\eqref{eq:20220915f} and $D_{h}(\bx_{k+1},\bz_{k+1}) \geq 0$, we obtain
\begin{align}
\label{eq:20220915e}
-\frac{1}{\alpha_k}D_{h}(\bx_{k},\bz_{k+1})  + f_{i_{k}}(\bx_{k})- f_{i_{k}}(\bxs) 
&\leq \frac{1}{\alpha_k} D_{h}(\bxs, \bx_{k}) - \frac{\left(1+\beta\alpha_k\right)}{\alpha_k} D_{h}(\bxs, \bx_{k+1}) \nonumber \\
&\qquad - \scalarp{\be_k , \bxs-\bx_{k}}.
\end{align}
Now, define $\EE_k[\cdot] = \EE[\cdot \,\vert\, \Fk_k]$,
where $\Fk_k$ is defined in Assumptions~~\myref{ass:variance} and is such that $\bx_k$ is $\Fk_k$-measurable and $i_k$ is independent of $\Fk_k$. Thus, taking the conditional expectation in~\eqref{eq:20220915e} and rearranging the terms, we have
\begin{align*}
     \left(1+\beta\alpha_k\right) \EE_k[D_{h}(\bxs, \bx_{k+1})] &\leq D_{h}(\bxs, \bx_{k})  - \alpha_k\EE_{k}\big[f_{i_{k}}(\bx_{k})- f_{i_{k}}(\bxs)\big]  + \EE_{k}[D_{h}(\bx_{k},\bz_{k+1})]\\
     &= D_{h}(\bxs, \bx_{k}) -\alpha_k (F(\bx_{k})- F(\bxs)) + \EE_{k}[D_{h}(\bx_{k},\bz_{k+1})].
\end{align*}
 Replacing $\bxs$ by $\bxs_*$
 and using Assumption~\eqref{ass:b2}, we get
\begin{align}
    \left(1+\beta\alpha_k\right) \EE_{k}[D_{h}(\bxs_*, \bx_{k+1})] &\leq D_{h}(\bxs_*, \bx_{k}) -\alpha_k(F(\bx_{k})- F(\bxs_*)) \nonumber \\
     &\qquad + \left(\alpakf \left (F(\bx_k) - F(\bxs_*)  \right) + \alpaks \sigma_k^2 + \alpha_k^2 N_{k}\right) \nonumber \\
     &= D_{h}(\bxs_*, \bx_{k}) 
     + \alpha_k^2 N_{k} \nonumber\\
     &\qquad -\alpha_k(1-\alpakff)(F(\bx_{k})- F(\bxs_*)) + \alpaks \sigma_k^2.  \nonumber \qedhere
\end{align}
\end{proof}
\begin{remark}\label{rmk:sgd}
We show in the next proposition that the same result in Proposition~\myref{prop:unified} with \(\beta = 0\) stands for Bregman SGD with the generic variance reduction term $\be_k$. As a consequence, since all the variance reduced results of the paper stem from Proposition~\myref{prop:unified} with \(\beta = 0\), they are also true for Bregman SGD with the same variance reduction techniques.
\end{remark}
\begin{proposition}[SGD case]\label{prop:unifiedexplicit}
Suppose that Assumptions~\myref{ass:variance} are verified, and that the sequence
$(\bx_k)_{k \in \N}$ is generated by the explicit version of Algorithm~\myref{algo:unified},i.e., SGD with the generic variance reduction term~\eqref{eq:sgdvr}. Then, for all $k \in \N$,
\begin{align}
     \EE[D_{h}(\bxs_*, \bx_{k+1})\,\vert\, \Fk_k] &\coloneqq \EE[D_{h}(\bxs_*, \bz_{k+1})\,\vert\, \Fk_k] \nonumber \\      &\leq D_{h}(\bxs_*, \bx_{k}) -\alpha_k\left[1-\alpha_k A_k\right]\left(F(\bx_{k})- F(\bxs_*)\right) 
     + \alpha_k^2 B_k \sigma_k^2  + \alpha_k^2 N_{k} \nonumber.
\end{align}
\end{proposition}
\begin{proof}
    We recall the definition of $\bx_{k+1}\coloneqq\bz_{k+1}$:
\begin{align}
    \bx_{k+1} &=\argmin_{\bxs \in \HH}\Big\{\underbrace{\scalarp{\bg_k, \bxs - \bx_k}-\scalarp{\be_k,\bxs-\bx_{k}}}_{=:R_{i_{k}}(\bxs)}  + \frac{1}{\alpha_k}D_{h}(\bxs,\bx_{k})\Big\} \nonumber \\
    &= \label{eq:20250219c2} 
    \nabla h^* \left(\nabla h (\bx_{k}) - \alpha_k \br_{k+1} \right), 
\end{align}
 where $\bg_k \in \partial f_{i_k}(\bx_k) $ and \(\br_{k+1} = \nabla R_{i_k}(\bx_{k+1}) \) such that \eqref{eq:20250219c2} holds. \\
Let $\bxs \in \mathcal{C}$. We have by convexity of \(R_{i_k}\):
\begin{align}\label{eq:20220915a2}
    \scalarp{\br_{k+1} , \bxs-\bx_{k+1}} &\leq R_{i_{k}}(\bxs)-R_{i_{k}}(\bx_{k+1}) \nonumber \\
    &=\scalarp{\bg_k, \bxs - \bx_k}-\scalarp{\bg_k, \bx_{k+1} - \bx_k} \nonumber \\
    &\qquad -\scalarp{\be_k , \bxs-\bx_{k}} + \scalarp{\be_k , \bx_{k+1}-\bx_{k}}.
\end{align}
Since \(\br_{k+1} = \frac{1}{\alpha_k} \left(\nabla h(\bx_{k})-\nabla h(\bx_{k+1})\right)\) by~\eqref{eq:20250219c2}, it follows from~\eqref{eq:20220915a2} and the convexity of $f_{i_k}$:
\begin{align}\label{eq:20251010a}
  &\frac{1}{\alpha_k}\scalarp{\nabla h(\bx_{k})-\nabla h(\bx_{k+1}) , \bxs-\bx_{k+1}} \nonumber \\
  &\leq f_{i_{k}}(\bxs) - f_{i_{k}}(\bx_k) -\scalarp{\bg_k, \bx_{k+1} - \bx_k}-\scalarp{\be_k , \bxs-\bx_{k}} + \scalarp{\be_k , \bx_{k+1}-\bx_{k}} \nonumber\\
  &= f_{i_{k}}(\bxs) - f_{i_{k}}(\bx_k) -\scalarp{\bv_k, \bx_{k+1} - \bx_k}-\scalarp{\be_k , \bxs-\bx_{k}}.
\end{align}
By using the three points identity in Lemma~\myref{lem:20250424a},
\(\scalarp{\nabla h(\bx_{k})-\nabla h(\bx_{k+1}) , \bxs-\bx_{k+1}}=D_{h}(\bx_{k+1},\bx_{k}) + D_{h}(\bxs, \bx_{k+1}) - D_{h}(\bxs, \bx_{k})\),
in the~\eqref{eq:20251010a}, we find
\begin{align}
\label{eq:20220915f2}
&\scalarp{\bv_k, \bx_{k+1}-\bx_{k}} + f_{i_{k}}(\bx_{k})- f_{i_{k}}(\bxs) + \frac{1}{\alpha_k}D_{h}(\bx_{k+1},\bx_{k})\nonumber \\
&\leq
         \frac{1}{\alpha_k} D_{h}(\bxs, \bx_{k}) - \frac{1}{\alpha_k} D_{h}(\bxs, \bx_{k+1}) - \scalarp{\be_k , \bxs-\bx_{k}}.
\end{align}
We recall that the explicit direction $\bv_k$ can be rewritten as 
\begin{align}\label{eq:10022025b2}
  \bv_k = \frac{1}{\alpha_k} \left(\nabla h(\bx_k) - \nabla h(\bz_{k+1})\right) = \frac{1}{\alpha_k} \left(\nabla h(\bx_k) - \nabla h(\bx_{k+1})\right).
\end{align} 
Using the definitions of $\bv_k$ in~\eqref{eq:10022025b2}, we can express the left hand side as follows:
\begin{align}
\scalarp{\bv_k , \bx_{k+1}-\bx_{k}} + f_{i_{k}}(\bx_{k})- f_{i_{k}}(\bxs) + \frac{1}{\alpha_k}D_{h}(\bx_{k+1},\bx_{k})
&=  \frac{1}{\alpha_k} \scalarp{ \nabla h(\bx_k) - \nabla h(\bx_{k+1}) , \bx_{k+1}-\bx_{k}} \nonumber \\
&\qquad +  \frac{1}{\alpha_k}D_{h}(\bx_{k+1},\bx_{k}) + f_{i_{k}}(\bx_{k})- f_{i_{k}}(\bxs) \nonumber \\
&= -\frac{1}{\alpha_k}D_{h}(\bx_{k},\bx_{k+1}) \nonumber\\
&\label{eq:20220915d2}\qquad + f_{i_{k}}(\bx_{k})- f_{i_{k}}(\bxs).
\end{align}
By using~\eqref{eq:20220915d2} in~\eqref{eq:20220915f2}, we obtain
\begin{align}
\label{eq:20220915e2}
-\frac{1}{\alpha_k}D_{h}(\bx_{k},\bx_{k+1})  + f_{i_{k}}(\bx_{k})- f_{i_{k}}(\bxs) 
&\leq \frac{1}{\alpha_k} D_{h}(\bxs, \bx_{k}) - \frac{1}{\alpha_k} D_{h}(\bxs, \bx_{k+1}) - \scalarp{\be_k , \bxs-\bx_{k}}.
\end{align}
Now, define $\EE_k[\cdot] = \EE[\cdot \,\vert\, \Fk_k]$,
where $\Fk_k$ is defined in Assumptions~~\myref{ass:variance} and is such that $\bx_k$ is $\Fk_k$-measurable and $i_k$ is independent of $\Fk_k$. Thus, taking the conditional expectation in~\eqref{eq:20220915e2} and rearranging the terms, we have
\begin{align*}
    \EE_k[D_{h}(\bxs, \bx_{k+1})] &\leq D_{h}(\bxs, \bx_{k})  - \alpha_k\EE_{k}\big[f_{i_{k}}(\bx_{k})- f_{i_{k}}(\bxs)\big]  + \EE_{k}[D_{h}(\bx_{k},\bx_{k+1})]\\
     &= D_{h}(\bxs, \bx_{k}) -\alpha_k (F(\bx_{k})- F(\bxs)) + \EE_{k}[D_{h}(\bx_{k},\bx_{k+1})].
\end{align*}
 Replacing $\bxs$ by $\bxs_*$
 and using Assumption~\eqref{ass:b2}, we get
\begin{align}
    \EE_{k}[D_{h}(\bxs_*, \bx_{k+1})] &\leq D_{h}(\bxs_*, \bx_{k}) -\alpha_k(F(\bx_{k})- F(\bxs_*)) \nonumber \\
     &\qquad + \left(\alpakf \left (F(\bx_k) - F(\bxs_*)  \right) + \alpaks \sigma_k^2 + \alpha_k^2 N_{k}\right) \nonumber \\
     &= D_{h}(\bxs_*, \bx_{k}) 
     + \alpha_k^2 N_{k} \nonumber\\
     &\qquad -\alpha_k(1-\alpakff)(F(\bx_{k})- F(\bxs_*)) + \alpaks \sigma_k^2.  \nonumber \qedhere
\end{align}
\end{proof}
\subsection{Technical lemmas}
These following technical lemmas are needed in the proofs.\begin{lemma}\label{prop:unified3}
Suppose that Assumptions~\myref{ass:variance} are verified and that the sequence
$(\bx_k)_{k \in \N}$ is generated by Algorithm~\myref{algo:unified} with $(\alpha_k)_{k \in \N}$ a non-decreasing positive real-valued sequence. Let $(M_k)_{k\in \N}$ be a non-increasing positive real-valued sequence. Then, for all $k \in \N$,
\begin{align}
    \frac{1}{\alpha_{k+1}^2}  \EE[D_{h}(\bxs_*, \bx_{k+1})] + M_{k+1}\EE[\sigma_{k+1}^2] 
    &\leq \frac{1}{\alpha_{k}^2} \EE[D_{h}(\bxs_*, \bx_{k})] + (M_k + B_k -\rho M_k)\EE[\sigma_k^2] \nonumber\\
    &\qquad  -\frac{1}{\alpha_{k}}\left(1-\alpha_k(A_k + M_k\ck)\right)\EE[F(\bx_k) - F(\bxs_*)]
    \nonumber\\
    &\qquad 
    + \EE[N_{k} ]. \nonumber
\end{align}
\end{lemma}
\begin{proof}
    Define $\EE_k[\cdot] = \EE[\cdot \,\vert\, \Fk_k]$. It follows from Proposition~\myref{prop:unified} that
    \begin{align}
        \frac{1}{\alpha_{k+1}^2}  \EE_{k}[D_{h}(\bxs_*, \bx_{k+1})] + M_{k+1}\EE_k[\sigma_{k+1}^2] &\leq \frac{1}{\alpha_{k}^2} \EE_{k}[D_{h}(\bxs_*, \bx_{k+1})] + M_k\EE_k[\sigma_{k+1}^2]  \nonumber \\
        &\leq \frac{1}{\alpha_{k}^2} D_{h}(\bxs_*, \bx_{k}) + N_{k} + B_k\sigma_k^2 + M_k\EE_k[\sigma_{k+1}^2] \nonumber\\
        &\qquad -\frac{1}{\alpha_{k}}(1-\alpakff)(F(\bx_{k})- F(\bxs_*)). \nonumber
    \end{align}
    By taking the total expectation, we get
    \begin{align}
        \frac{1}{\alpha_{k+1}^2}  \EE[D_{h}(\bxs_*, \bx_{k+1})] + M_{k+1}\EE[\sigma_{k+1}^2] &\leq \frac{1}{\alpha_{k}^2} \EE[D_{h}(\bxs_*, \bx_{k})] + \EE[N_{k}] + B_k\EE[\sigma_k^2] + M_k\EE[\sigma_{k+1}^2] \nonumber\\
        &\qquad -\frac{1}{\alpha_{k}}(1-\alpakff)\EE[F(\bx_{k})- F(\bxs_*)] \nonumber \\
        &\leq \frac{1}{\alpha_{k}^2} \EE[D_{h}(\bxs_*, \bx_{k})] + \EE[N_{k}] + B_k\EE[\sigma_k^2] \nonumber\\
        &\qquad -\frac{1}{\alpha_{k}}(1-\alpakff)\EE[F(\bx_{k})- F(\bxs_*)] \nonumber \\
        &\qquad + M_k(1-\rho) \EE\left[\sigma_k^2\right] + M_k\ck\EE[F(\bx_k) - F(\bxs_*)] \nonumber \\
        &= \frac{1}{\alpha_{k}^2} \EE[D_{h}(\bxs_*, \bx_{k})] + \EE[N_{k}] + (M_k + B_k -\rho M)\EE[\sigma_k^2] \nonumber\\
        &\qquad -\frac{1}{\alpha_{k}}\left(1-\alpha_k(A_k + M_k\ck)\right)\EE[F(\bx_k) - F(\bxs_*)]. \nonumber \qedhere
    \end{align}
\end{proof}

\begin{lemma}\label{lem:bregsmooth} \emph{\citep[Lemma 3]{pmlr-v139-dragomir21a}.} If a convex function $f$ is $L$-relatively smooth w.r.t. $h$, then for any $\eta \le 
\frac{1}{L}$ and $\bxs, \bys \in \operatorname{int} \mathcal{C}$,
\begin{equation*}
    D_f(\bxs,\bys) \ge \frac{1}{\eta} D_{h^*} \left( \nabla h(\bxs) - \eta (\nabla f(\bxs) - \nabla f(\bys)), \nabla h(\bxs) \right).
\end{equation*}
\end{lemma}
\begin{lemma}\label{lem:20250304a} \emph{\citep[Lemma 2]{pmlr-v139-dragomir21a}.} Let $\bxs \in \operatorname{int} \mathcal{C}$, and $g_1, g_2 \in \R^d$. Define the points $\bxs_1^+, \bxs_2^+, \bxs^+$ as the unique points satisfying $\nabla h(\bxs_1^+) = \nabla h(\bxs) - g_1$, $\nabla h(\bxs_2^+) = \nabla h(\bxs) - g_2$, $\nabla h(\bxs^+) = \nabla h(\bxs) - \frac{g_1 + g_2}{2}.$
Then 
\begin{align*}
D_h(\bxs, \bxs^+) &\leq \frac{1}{2} \left[D_h(\bxs, \bxs_1^+) + D_h(\bxs, \bxs_2^+)\right] \\
                &= \frac{1}{2} \Big[D_{h^*}\left(\nabla h(\bxs_1^+), \nabla h(\bxs)\right) + D_{h^*}\left(\nabla h(\bxs_2^+), \nabla h(\bxs)\right)\Big].
\end{align*}
For the Euclidean case, i.e. $h = \|\cdot \|^2$, we obtain the standard inequality $\|\frac{g_1 + g_2}{2}\|^2 \leq \frac{1}{2} \left(\|g_1\|^2 + \|g_2\|^2\right)$.
\end{lemma}
Now we recall the Bregman version  of the following Euclidean variance decomposition for a r.v. $\bx$:
\begin{equation*}
    \EE\|\bx\|^2 = \left\|\EE[\bx]\right\|^2 + \EE\|\bx - \EE[\bx]\|^2.
\end{equation*}
\begin{lemma}[Bregman variance decomposition by {\citep[Theorem 2.3]{pfau2025generalized}}]\label{lem:vardec}
    Let $\bx$ be a random variable on \(\dom h^*\) such that $ \EE[\bx] \in \interior\dom h^*$. Then for any $\bus \in \interior\dom h^*$,
    \begin{equation}
    \EE[D_{h^*}(\bx, \bus)] = D_{h^*}(\EE[\bx], \bus) + \EE[D_{h^*}(\bx, E[\bx])]. 
    \end{equation}
\end{lemma}

\subsection{Proofs of Section~\myref{sect:analysis}}
\label{app:proofsect3}

\vspace{2ex}
\noindent
\textbf{Proof of Theorem~\myref{theo:unifiedmainonlyconvex}.} From Lemma~\myref{prop:unified3} and since $B_k - \rho M_k \leq 0$ , we get
    \begin{align}
        (1/\alpha_{k})\left(1-\alpha_k(A_k + M_k\ck)\right)\EE[F(\bx_k) - F(\bxs_*)] &\leq \EE\left[\frac{1}{\alpha_{k}^2}  D_{h}(\bxs_*, \bx_{k}) + M_{k}\sigma_{k}^2\right]  \nonumber \\
        &\qquad - \EE\left[\frac{1}{\alpha_{k+1}^2}  D_{h}(\bxs_*, \bx_{k+1}) + M_{k+1}\sigma_{k+1}^2\right] \nonumber \\
        &\qquad + \EE [N_{k}] \nonumber .
    \end{align}
    That means that, by summing on $t=0, \dots, k-1$, we obtain
    \begin{align}
        \EE[F(\bbar{\bx}_k) - F(\bxs_*)] &\leq \frac{(1/\alpha_{0}^2) \EE[ D_{h}(\bxs_*, \bx_{0})] + M_{0}\EE[\sigma_{0}^2]}{\sum_{t=0}^{k-1} (1/\alpha_{t})\left(1-\alpha_t(A_t + M_t C_t)\right)}\nonumber \\
        &\qquad + \sum_{t=0}^{k-1}\frac{\EE[N_t]}{\sum_{t=0}^{k-1} (1/\alpha_{t})\left(1-\alpha_t(A_t + M_t C_t)\right)}, \nonumber
    \end{align}
    where $\displaystyle \bbar{\bx}_k = \sum_{t=0}^{k-1} \frac{(1/\alpha_{t})\left(1-\alpha_t(A_t + M_t C_t)\right)}{\sum_{t=0}^{k-1} (1/\alpha_{t})\left(1-\alpha_t(A_t + M_t C_t)\right)} \bx_t$. \hfill \qed 
    
\vspace{2ex}
\noindent
\textbf{Proof of Theorem~\myref{theo:strongconvex}.} From Lemma~\myref{prop:unified3} and relative strong convexity of $F$, we have
    \begin{align}
        \frac{1}{\alpha_{k+1}^2}  \EE[D_{h}(\bxs_*, \bx_{k+1})] + M_{k+1}\EE[\sigma_{k+1}^2] 
        &\leq \frac{1}{\alpha_{k}^2} \EE[D_{h}(\bxs_*, \bx_{k})] +  \EE[N_{k}]+ (M_k + B_k -\rho M_k)\EE[\sigma_k^2] \nonumber\\
        &\qquad -\frac{1}{\alpha_{k}}\left(1-\alpha_k(A_k + M_k\ck)\right)\EE[D_F(\bx_{k}, \bxs_*)] \nonumber \\
        &\leq \frac{1}{\alpha_{k}^2} \EE[D_{h}(\bxs_*, \bx_{k})] + M_k\left[1+\frac{B_k}{M_k}-\rho\right]\EE[\sigma_k^2] +  \EE[N_{k}] \nonumber \\
        &\qquad -\frac{\mu}{\alpha_{k}}\left(1-\alpha_k(A_k + M_k\ck)\right)\EE[D_h(\bx_k, \bxs_*)] \nonumber \\
        &\leq q_k \left(\frac{1}{\alpha_{k}^2} \EE[D_{h}(\bxs_*, \bx_{k})] + M_k\EE[\sigma_k^2]\right) + \EE[N_{k}] \nonumber, 
    \end{align}
    with $q_k = \max\left\{1 - \alpha_k\gamma_h\mu\left(1-\alpha_k (A_k+ M_k C_k)\right), 1+\frac{B_k}{M_k}-\rho\right\}$. \hfill \qed 
\subsection{Proofs of Section~\myref{sect:applications}}\label{app:proofsect4}
\noindent
\textbf{Proof of Theorem~\myref{theo:bsppa}.} From Proposition~\myref{prop:unified}, we have
        \begin{equation}\label{eq:20250219a}
            \alpha_k \EE[F(\bx_{k})- F(\bxs_*)] \leq \EE[D_{h}(\bxs_*, \bx_{k})] - \EE[D_{h}(\bxs_*, \bx_{k+1})] + \alpha_k^2 \sigma^2_{*}.
        \end{equation}
        Let $k \geq 1$. Summing from $0$ up to $k-1$ and dividing both side by $\sum_{t=0}^{k-1}\alpha_t$,~\eqref{eq:20250219a} gives
        \begin{align}
            \sum_{t=0}^{k-1}\frac{\alpha_t}{\sum_{t=0}^{k-1}\alpha_t}\EE[F(\bx_{t})- F(\bxs_*)] &\leq \frac{1}{\sum_{t=0}^{k-1}\alpha_t} \left(D_{h}(\bxs_*, \bx_{0})  - \EE[D_{h}(\bxs_*, \bx_{k})]\right) + \sigma^2_{*} \frac{\sum_{t=0}^{k-1}\alpha_t^2}{\sum_{t=0}^{k-1}\alpha_t} \nonumber \\
            &\leq \frac{D_{h}(\bxs_*, \bx_{0})}{\sum_{t=0}^{k-1}\alpha_t} + \sigma^2_{*} \frac{\sum_{t=0}^{k-1}\alpha_t^2}{\sum_{t=0}^{k-1}\alpha_t}. \nonumber
        \end{align}
        Using convexity and Jensen inequality, we finally obtain the result. \hfill \qed 
        
\vspace{2ex}
\noindent
\textbf{Proof of Theorem~\myref{theo:bsppa2}.} Proposition~\myref{prop:unified} gives
        \begin{align}
            \left(1+\beta\alpha_k\right) \EE[D_{h}(\bxs_*, \bx_{k+1})]  &\leq 
             \EE[D_{h}(\bxs_*, \bx_{k})] -\alpha_k\EE[F(\bx_{k})- F(\bxs_*)]  + \alpha_k^2 \sigma^2_{*} \nonumber \\
             &\leq \EE[D_{h}(\bxs_*, \bx_{k})] + \alpha_k^2 \sigma^2_{*} \nonumber.
        \end{align}
        So
        \begin{align}
            \EE[D_{h}(\bxs_*, \bx_{k+1})]  &\leq \frac{1}{1+\beta\alpha_k} \EE[D_{h}(\bxs_*, \bx_{k})] + \frac{\alpha_k^2}{1+\beta\alpha_k} \sigma^2_{*}. \nonumber
        \end{align}
        Let $\alpha_k = \alpha$ and $q = \frac{1}{1+\beta\alpha}$. Summing from $0$ up to $k-1$, we obtain
        \begin{align}
            \EE[D_{h}(\bxs_*, \bx_{k})]  &\leq q^k \EE[D_{h}(\bxs_*, \bx_{0})] + q^k \sigma^2_{*} \alpha^2 \sum_{t=0}^{k-1}q^{-t} \nonumber \\
            &= q^k \EE[D_{h}(\bxs_*, \bx_{0})] +  \sigma^2_{*} \alpha^2 \sum_{t=0}^{k-1}q^{k-t} \nonumber \\
            &\leq q^k \EE[D_{h}(\bxs_*, \bx_{0})] +  \alpha^2\frac{1}{1-q} \sigma^2_{*}, \nonumber 
        \end{align}
which gives the results. \hfill \qed 

\vspace{2ex}
\noindent
\textbf{Proof of Lemma~\myref{lem:20250311b}.} From Equation~\eqref{eq:10022025a}, we infer that $\nabla h(\bz_{k+1}) = \nabla h(\bx_{k}) - \alpha_k \nabla f_{i_k}(\bx_{k}) + \alpha_k \be_k$. Set $\nabla h(\bxs_1^+) = \nabla h(\bx_{k}) - 2\alpha_k [\nabla f_{i_k}(\bx_{k}) - \nabla f_{i_k}(\bxs_{*})]$ and $\nabla h(\bxs_2^+) = \nabla h(\bx_{k}) - 2\alpha_k (\nabla f_{i_k}(\bxs_{*}) - \be_k)$.
Then, from Lemma~\myref{lem:20250304a}, we know that $D_h(\bx_{k}, \bz_{k+1}) \leq (S_1 + S_2)/2$, where
    \begin{align*}
    S_1 &= D_{h^*}(\nabla h(\bx_{k}) - 2\alpha_k [\nabla f_{i_k}(\bx_{k}) - \nabla f_{i_k}(\bxs_{*})], \nabla h(\bx_{k})), \\
    S_2 &= D_{h^*}(\nabla h(\bx_{k}) - 2\alpha_k (\nabla f_{i_k}(\bxs_{*}) - \be_k), \nabla h(\bx_{k})).
    \end{align*}
    
    Using the gain function in Equation~\eqref{eq:gain}, then Lemma~\myref{lem:bregsmooth}, and finally Equation~\eqref{eq:gain2}, we have
    \begin{align*}
    \EE_k[S_1] &= \EE_k[D_{h^{*}}(\nabla h(\bx_{k}) - 2\alpha_k (\nabla f_{i_k}(\bx_{k}) - \nabla f_{i_k}(\bxs_{*})), \nabla h(\bx_{k}))] \\
    &\leq 4 L^2 \alpha_k^2 \EE_k \left[G \left(\nabla h(\bx_{k}), \nabla h(\bx_{k}), \frac{1}{L} (\nabla f_{i_k}(\bx_{k}) - \nabla f_{i_k}(\bxs_{*})) \right) \times \right. \\
    &\qquad \left. D_{h^{*}}\left(\nabla h(\bx_{k}) - \frac{1}{L} (\nabla f_{i_k}(\bx_{k}) - \nabla f_{i_k}(\bxs_{*})), \nabla h(\bx_{k}) \right) \right] \\
    &\leq 4 L \alpha_k^2 G_k\EE_k\left[D_{h^{*}}\left(\nabla h(\bx_{k}) - \frac{1}{L} (\nabla f_{i_k}(\bx_{k}) - \nabla f_{i_k}(\bxs_{*})), \nabla h(\bx_{k}) \right)\right] \\
    &\leq 4 L \alpha_k^2 G_k D_F(\bx_{k}, \bxs_{*}).
    \end{align*}
    We know that $- 2\alpha_k (\nabla f_{i_k}(\bxs_{*}) - \be_k) = \zeta_k - \EE_k[\zeta_k]$. From Lemma~\myref{lem:vardec}, Equation~\eqref{eq:gain}, and Equation~\eqref{eq:gain3}, it follows that
    \begin{align*}
    \EE_k[S_2] &= \EE_k[D_{h^*}(\nabla h(\bx_{k}) + \zeta_k - \EE_k[\zeta_k], \nabla h(\bx_{k}))] \\
    &= \EE_k[D_{h^*}(\nabla h(\bx_{k}) + \zeta_k - \EE_k[\zeta_k], \nabla h(\bx_{k}) - \EE_k[\zeta_k])] \\
    &\qquad - D_{h^*}(\nabla h(\bx_{k}), \nabla h(\bx_{k}) - \EE_k[\zeta_k]) \\
    &\leq 4 L^2 \alpha_k^2 \EE_k\Bigg[G \bigg(\nabla h(\bx_{k}) - 2\alpha_k \frac{1}{n} \sum_{i=1}^n \nabla f_i (\bphi_i^k), \nabla h(\bphi_{i_k}^k), \frac{1}{L} \left(\nabla f_{i_k}(\bphi_{i_k}^k) - \nabla f_{i_k}(\bxs_{*})\right) \bigg) \times \\
    &\qquad  D_{h^*} \bigg(\nabla h(\bphi_{i_k}^k) - \frac{1}{L} \left(\nabla f_{i_k}(\bphi_{i_k}^k) - \nabla f_{i_k}(\bxs_{*})\right), \nabla h(\bphi_{i_k}^k) \bigg)\Bigg] \\
    &\qquad - \EE_k[D_{h^*}(\nabla h(\bx_{k}), \nabla h(\bx_{k}) - \EE_k[\zeta_k])] \\
    &\leq 2 \alpha_k^2 G_k \sigma_k^2 - D_{h^*}(\nabla h(\bx_{k}), \nabla h(\bx_{k}) - \EE_k[\zeta_k]).
    \end{align*}
    Putting all together, we get Equation~\eqref{eq:varsapa}. For Equation~\eqref{eq:20221004c}, we have
\begin{align*}
    \sigma_{k+1}^2 &= 2L^2 \EE_{k+1} \left[ D_{h^*} \left(\nabla h\left(\bphi_{i_{k+1}}^{k+1}\right) - \frac{1}{L} \left(\nabla f_{i_{k+1}}\left(\bphi_{i_{k+1}}^{k+1}\right) - \nabla f_{i_{k+1}}(\bxs_{*})\right), \nabla h\left(\bphi_{i_{k+1}}^{k+1}\right) \right)\right] \\
    &= 2L^2 \frac{1}{ n} \sum_{i=1}^{n} D_{h^*} \left(\nabla h\left(\bphi_{i}^{k+1}\right) - \frac{1}{L} \left(\nabla f_{i}\left(\bphi_{i}^{k+1}\right) - \nabla f_{i}(\bxs_{*})\right), \nabla h\left(\bphi_{i}^{k+1}\right) \right) \\
    &= 2L^2 \frac{1}{ n} \sum_{i=1}^{n} D_{h^*} \bigg(\nabla h\left(\bphi^k_i + \delta_{i,i_k} (\bx_{k} - \bphi^k_i)\right) - \frac{1}{L} \left(\nabla f_{i}\left(\bphi^k_i + \delta_{i,i_k} (\bx_{k} - \bphi^k_i)\right) - \nabla f_{i}(\bxs_{*})\right),  \\
    &\qquad \qquad \qquad   \nabla h\left(\bphi^k_i + \delta_{i,i_k} (\bx_{k} - \bphi^k_i)\right) \bigg).
\end{align*}
That means
\begin{align*}
&\EE_{k}\left[\sigma^2_{k+1}\right] \\
&= 2L^2 \frac{1}{ n} \sum_{i=1}^{n} \EE_{k} \Bigg[ D_{h^*} \bigg(\nabla h\left(\bphi^k_i + \delta_{i,i_k} (\bx_{k} - \bphi^k_i)\right) - \frac{1}{L} \left(\nabla f_{i}\left(\bphi^k_i + \delta_{i,i_k} (\bx_{k} - \bphi^k_i)\right) - \nabla f_{i}(\bxs_{*})\right), \\
    &\qquad \qquad \qquad   \nabla h\left(\bphi^k_i + \delta_{i,i_k} (\bx_{k} - \bphi^k_i)\right) \bigg) \Bigg]\\
&= 2L^2 \frac{1}{ n} \sum_{i=1}^{n} \frac{1}{ n} \sum_{j=1}^{n}  D_{h^*} \bigg(\nabla h\left(\bphi^k_i + \delta_{i,j} (\bx_{k} - \bphi^k_i)\right) - \frac{1}{L} \left(\nabla f_{i}\left(\bphi^k_i + \delta_{i,j} (\bx_{k} - \bphi^k_i)\right) - \nabla f_{i}(\bxs_{*})\right), \\
    &\qquad \qquad \qquad  \nabla h\left(\bphi^k_i + \delta_{i,j} (\bx_{k} - \bphi^k_i)\right) \bigg) \\
&= 2L^2 \frac{1}{ n} \sum_{i=1}^{n} \left[ \frac{n-1}{ n}  D_{h^*} \left(\nabla h\left(\bphi^k_i\right) - \frac{1}{L} \left(\nabla f_{i}\left(\bphi^k_i\right) - \nabla f_{i}(\bxs_{*})\right), \nabla h\left(\bphi^k_i\right) \right) \right. \\
&\qquad \qquad \qquad \left. + \frac{1}{n} D_{h^*} \left(\nabla h\left(\bx_{k}\right) - \frac{1}{L} \left(\nabla f_{i}\left(\bx_{k}\right) - \nabla f_{i}(\bxs_{*})\right), \nabla h\left(\bx_{k}\right) \right)\right] \\
&= \left(1-\frac{1}{n}\right)2L^2 \frac{1}{ n} \sum_{i=1}^{n}  D_{h^*} \left(\nabla h\left(\bphi^k_i\right) - \frac{1}{L} \left(\nabla f_{i}\left(\bphi^k_i\right) - \nabla f_{i}(\bxs_{*})\right), \nabla h\left(\bphi^k_i\right) \right) \\
&\qquad \qquad \qquad + 2L^2 \frac{1}{n^2} \sum_{i=1}^{n} D_{h^*} \left(\nabla h\left(\bx_{k}\right) - \frac{1}{L} \left(\nabla f_{i}\left(\bx_{k}\right) - \nabla f_{i}(\bxs_{*})\right), \nabla h\left(\bx_{k}\right) \right) \\
&\leq  \left(1-\frac{1}{n}\right) \sigma^2_{k} + \frac{2L}{n} D_F(\bx_{k}, \bxs_{*}).
\end{align*}
The last inequality comes from Lemma~\myref{lem:bregsmooth}. \hfill \qed 
        
\vspace{2ex}
\noindent
\textbf{Proof of Lemma~\myref{lem:20250311a}.} From Lemma~\myref{lem:20250304a}, we know that $D_h(\bx_{k}, \bz_{k+1}) \leq (S_1 + S_2)/2$, where
    \begin{align*}
    S_1 &= D_{h^*}(\nabla h(\bx_{k}) - 2\alpha_k [\nabla f_{i_k}(\bx_{k}) - \nabla f_{i_k}(\bxs_{*})], \nabla h(\bx_{k})), \\
    S_2 &= D_{h^*}(\nabla h(\bx_{k}) - 2\alpha_k (\nabla f_{i_k}(\bxs_{*}) - \be_k), \nabla h(\bx_{k})).
    \end{align*}
    Using the gain function in Equation~\eqref{eq:gain}, then Lemma~\myref{lem:bregsmooth}, and finally Equation~\eqref{eq:gain4}, we have
    \begin{align*}
    \EE_k[S_1] &= \EE_k[D_{h^{*}}(\nabla h(\bx_{k}) - 2\alpha_k (\nabla f_{i_k}(\bx_{k}) - \nabla f_{i_k}(\bxs_{*})), \nabla h(\bx_{k}))] \\
    &\leq 4 L^2 \alpha_k^2 \EE_k \left[G \left(\nabla h(\bx_{k}), \nabla h(\bx_{k}), \frac{1}{L} (\nabla f_{i_k}(\bx_{k}) - \nabla f_{i_k}(\bxs_{*})) \right) \times \right. \\
    &\qquad \left. D_{h^{*}}\left(\nabla h(\bx_{k}) - \frac{1}{L} (\nabla f_{i_k}(\bx_{k}) - \nabla f_{i_k}(\bxs_{*})), \nabla h(\bx_{k}) \right) \right] \\
    &\leq 4 L \alpha_k^2 G_k\EE_k\left[D_{h^{*}}\left(\nabla h(\bx_{k}) - \frac{1}{L} (\nabla f_{i_k}(\bx_{k}) - \nabla f_{i_k}(\bxs_{*})), \nabla h(\bx_{k}) \right)\right] \\
    &\leq 4 L \alpha_k^2 G_k D_F(\bx_{k}, \bxs_{*}).
    \end{align*}
    We know that $- 2\alpha_k (\nabla f_{i_k}(\bxs_{*}) - \be_k) = \zeta_k - \EE_k[\zeta_k]$. From Lemma~\myref{lem:vardec}, Equation~\eqref{eq:gain}, and Equation~\eqref{eq:gain5}, it follows that
    \begin{align*}
    \EE_k[S_2] &= \EE_k[D_{h^*}(\nabla h(\bx_{k}) + \zeta_k - \EE_k[\zeta_k], \nabla h(\bx_{k}))] \\
    &= \EE_k[D_{h^*}(\nabla h(\bx_{k}) + \zeta_k - \EE_k[\zeta_k], \nabla h(\bx_{k}) - \EE_k[\zeta_k])] \\
    &\qquad - D_{h^*}(\nabla h(\bx_{k}), \nabla h(\bx_{k}) - \EE_k[\zeta_k]) \\
    &\leq 4 L^2 \alpha_k^2 \EE_k\left[G \left(\nabla h(\bx_{k}) - 2\alpha_k \nabla F (\bu_{k}), \nabla h(\bu_{k}), \frac{1}{L} \left(\nabla f_{i_k}(\bu_{k}) - \nabla f_{i_k}(\bxs_{*})\right) \right) \times \right. \\
    &\qquad \left. D_{h^*} \left(\nabla h(\bu_k) - \frac{1}{L} \left(\nabla f_{i_k}(\bu_{k}) - \nabla f_{i_k}(\bxs_{*})\right), \nabla h(\bu_{k}) \right)\right] \\
    &\qquad - \EE_k[D_{h^*}(\nabla h(\bx_{k}), \nabla h(\bx_{k}) - \EE_k[\zeta_k])] \\
    &\leq 2 \alpha_k^2 G_k \sigma_k^2 - D_{h^*}(\nabla h(\bx_{k}), \nabla h(\bx_{k}) - \EE_k[\zeta_k]).
    \end{align*}
    Putting all together, we get Equation~\eqref{eq:varlsvrp}. For Equation~\eqref{eq:20221004c}, we have
\begin{align*}
    \sigma_{k+1}^2 &=  2 L^2 \EE_{k+1} \Bigg[ D_{h^*} \Big(\nabla h(\bu_{k+1}) - \frac{1}{L} (\nabla f_{i_{k+1}}(\bu_{k+1}) - \nabla f_{i_{k+1}}(\bxs_{*})), \nabla h(\bu_{k+1}) \Big)\Bigg],\\
    &=2L^2 \EE_{k+1} \Bigg[ D_{h^*} \Big(\nabla h\left((1-\varepsilon_k)\bu_k + \varepsilon_k \bx_k\right) - \frac{1}{L} \left(\nabla f_{i_{k+1}}\left((1-\varepsilon_k)\bu_k + \varepsilon_k \bx_k\right) - \nabla f_{i_{k+1}}(\bxs_{*})\right), \\
    &\qquad \qquad \qquad    \nabla h\left((1-\varepsilon_k)\bu_k + \varepsilon_k \bx_k\right) \Big)\Bigg] \\
    &= 2L^2 \frac{1}{ n} \sum_{i=1}^{n} D_{h^*} \Big(\nabla h\left((1-\varepsilon_k)\bu_k + \varepsilon_k \bx_k\right) - \frac{1}{L} \left(\nabla f_{i}\left((1-\varepsilon_k)\bu_k + \varepsilon_k \bx_k\right) - \nabla f_{i}(\bxs_{*})\right), \\
    &\qquad \qquad \qquad \qquad   \nabla h\left((1-\varepsilon_k)\bu_k + \varepsilon_k \bx_k\right) \Big).
\end{align*}
That means
\begin{align*}
\EE_{k}\left[\sigma^2_{k+1}\right] &= 2L^2 \frac{1}{ n} \sum_{i=1}^{n} \EE_{k} \Bigg[D_{h^*} \Big(\nabla h\left((1-\varepsilon_k)\bu_k + \varepsilon_k \bx_k\right) - \frac{1}{L} \left(\nabla f_{i}\left((1-\varepsilon_k)\bu_k + \varepsilon_k \bx_k\right) - \nabla f_{i}(\bxs_{*})\right), \\
    &\qquad \qquad \qquad \qquad   \nabla h\left((1-\varepsilon_k)\bu_k + \varepsilon_k \bx_k\right) \Big)\Bigg] \\
& = (1-p) \sigma^2_{k} + 2 p L^2\frac{1}{ n} \sum_{i=1}^{n}D_{h^*} \left(\nabla h(\bx_k) - \frac{1}{L} \left(\nabla f_{i}( \bx_k) - \nabla f_{i}(\bxs_{*})\right), \nabla h( \bx_k) \right) \\
& \leq (1-p) \sigma^2_{k} + 2 p L D_F(\bx_{k}, \bxs_{*}).
\end{align*}
The last inequality comes from Lemma~\myref{lem:bregsmooth}. \hfill \qed 

\vspace{2ex}
\noindent
\textbf{Proof of Lemma~\myref{lem:20220927a}.} Thanks to Lemma~\myref{lem:20250304a}, we know that $D_h(\bx_{k}, \bz_{k+1}) \leq (S_1 + S_2)/2$, where
    \begin{align*}
    S_1 &= D_{h^*}(\nabla h(\bx_{k}) - 2\alpha_k [\nabla f_{i_k}(\bx_{k}) - \nabla f_{i_k}(\bxs_{*})], \nabla h(\bx_{k})), \\
    S_2 &= D_{h^*}(\nabla h(\bx_{k}) - 2\alpha_k (\nabla f_{i_k}(\bxs_{*}) - \be_{k}), \nabla h(\bx_{k})).
    \end{align*}
     Using the gain function in Equation~\eqref{eq:gain}, then Lemma~\myref{lem:bregsmooth}, and finally Equation~\eqref{gain6}, we have
    \begin{align*}
    \EE[S_1\,\vert\, \Fk_{s,k}] &= \EE[D_{h^{*}}(\nabla h(\bx_{k}) - 2\alpha_k (\nabla f_{i_k}(\bx_{k}) - \nabla f_{i_k}(\bxs_{*})), \nabla h(\bx_{k}))\,\vert\, \Fk_{s,k}] \\
    &\leq 4 L^2 \alpha_k^2 \EE \left[G \left(\nabla h(\bx_{k}), \nabla h(\bx_{k}), \frac{1}{L} (\nabla f_{i_k}(\bx_{k}) - \nabla f_{i_k}(\bxs_{*})) \right) \times \right. \\
    &\qquad \left. D_{h^{*}}\left(\nabla h(\bx_{k}) - \frac{1}{L} (\nabla f_{i_k}(\bx_{k}) - \nabla f_{i_k}(\bxs_{*})), \nabla h(\bx_{k}) \right) \,\vert\, \Fk_{s,k} \right] \\
    &\leq 4 L \alpha_k^2G_{sm + k} \EE \left[D_{h^{*}}\left(\nabla h(\bx_{k}) - \frac{1}{L} (\nabla f_{i_k}(\bx_{k}) - \nabla f_{i_k}(\bxs_{*})), \nabla h(\bx_{k}) \right) \,\vert\, \Fk_{s,k} \right] \\
    &\leq 4 L \alpha_k^2 G_{sm + k} D_F(\bx_{k}, \bxs_{*}).
    \end{align*}
    We know that $- 2\alpha_k (\nabla f_{i_k}(\bxs_{*}) - \be_k) = \zeta_k - \EE[\zeta_k\,\vert\, \Fk_{s,k}]$. From Lemma~\myref{lem:vardec}, Equation~\eqref{eq:gain}, and Equation~\eqref{gain7}, it follows that
    \begin{align*}
    \EE[S_2\,\vert\, \Fk_{s,k}] &= \EE[D_{h^*}(\nabla h(\bx_{k}) + \zeta_k - \EE[\zeta_k\,\vert\, \Fk_{s,k}], \nabla h(\bx_{k}))\,\vert\, \Fk_{s,k}] \\
    &= \EE[D_{h^*}(\nabla h(\bx_{k}) + \zeta_k - \EE[\zeta_k\,\vert\, \Fk_{s,k}], \nabla h(\bx_{k}) - \EE[\zeta_k\,\vert\, \Fk_{s,k}])\,\vert\, \Fk_{s,k}] \\
    &\qquad - D_{h^*}(\nabla h(\bx_{k}), \nabla h(\bx_{k}) - \EE[\zeta_k\,\vert\, \Fk_{s,k}]) \\
    &\leq 4 L^2 \alpha_k^2 \EE \left[G \left(\nabla h(\bx_{k}) - 2\alpha_k \nabla F(\tilde{\bx}_{s}), \nabla h(\tilde{\bx}_{s}), \frac{1}{L} (\nabla f_{i_k}(\tilde{\bx}_{s}) - \nabla f_{i_k}(\bxs_{*})) \right) \times \right. \\
    &\qquad \left. D_{h^*} \left(\nabla h(\tilde{\bx}_{s}) - \frac{1}{L} (\nabla f_{i_k}(\tilde{\bx}_{s}) - \nabla f_{i_k}(\bxs_{*})), \nabla h(\tilde{\bx}_{s}) \right) \,\vert\, \Fk_{s,k} \right] \\
    &\qquad - D_{h^*}(\nabla h(\bx_{k}), \nabla h(\bx_{k}) - \EE[\zeta_k\,\vert\, \Fk_{s,k}]) \\
    &\leq 2 \alpha_k^2 G_{sm + k} \sigma_k^2 - D_{h^*}(\nabla h(\bx_{k}), \nabla h(\bx_{k}) - \EE[\zeta_k\,\vert\, \Fk_{s,k}]). 
\end{align*}
Putting all together give the results. \hfill \qed 

\vspace{2ex}
\noindent
\textbf{Proof of Theorem~\myref{theo:bsvrp}.} Proposition~\myref{prop:unified} and Lemma~\myref{lem:vardec} give 
    \begin{align}
         \EE[D_{h}(\bxs_*, \bx_{k+1})\,\vert\, \Fk_{s,k}]  &\leq 
         D_{h}(\bxs_*, \bx_{k}) + \alpha_k^2 B_k \sigma_k^2 \nonumber \\
         &\qquad -\alpha_k\left[1-\alpha_k A_k\right]\EE[F(\bx_{k})- F(\bxs_*)\,\vert\, \Fk_{s,k}]\nonumber \\
         &= D_{h}(\bxs_*, \bx_{k}) -\alpha_k\left[1-\alpha_k A_k\right]\EE[F(\bx_{k})- F(\bxs_*)] \nonumber \\
         &\qquad + 2 L^2 \alpha_k^2 B_k \EE\left[D_{h^*} \left(\nabla h(\tilde{\bx}_{s}) - \frac{1}{L} (\nabla f_{i_k}(\tilde{\bx}_{s}) - \nabla f_{i_k}(\bxs_{*})), \nabla h(\tilde{\bx}_{s}) \right) \,\vert\, \Fk_{s,k} \right] \nonumber \\
         &\leq D_{h}(\bxs_*, \bx_{k}) -\alpha_k\left[1-\alpha_k A_k\right]\EE[F(\bx_{k})- F(\bxs_*)] \nonumber \\
         &\qquad + 2 L \alpha_k^2 B_k D_F(\tilde{\bx}_{s}, \bxs_{*}).
    \end{align}
    In the last inequality, we used Lemma~\myref{lem:bregsmooth}. By taking total expectation, it follows
    \begin{align}
         \EE[D_{h}(\bxs_*, \bx_{k+1})]  &\leq 
         \EE[D_{h}(\bxs_*, \bx_{k})] + 2\alpha_k^2 B_k L \EE[D_F(\tilde{\bx}_{s}, \bxs_{*})] \nonumber \\
         &\qquad -\alpha_k\left[1-\alpha_k A_k\right]\EE[F(\bx_{k})- F(\bxs_*)]\nonumber.%
    \end{align}
    Then by summing over $k=0, \ldots, m-1$ and recalling Remark~\myref{rmk:constsvrg}, we obtain
    \begin{align}
         \EE[D_{h}(\bxs_*, \bx^{m})] + \alpha\left[1-\alpha A \right]m\EE[D_F(\tilde{\bx}_{s+1}, \bxs_{*})] &\leq 
         \EE[D_{h}(\bxs_*, \bx_{0})] + 2\alpha^2 B L m \EE[D_F(\tilde{\bx}_{s}, \bxs_{*})]\nonumber \\ 
         &= \EE[D_{h}(\bxs_*, \tilde{\bx}_{s})] + 2\alpha^2 B L m \EE[D_F(\tilde{\bx}_{s}, \bxs_{*})]\nonumber \\ 
         &\leq \frac{1}{\mu}\EE[D_{F}(\bxs_*, \tilde{\bx}_{s})] + 2\alpha^2 B L m\EE[D_F(\tilde{\bx}_{s}, \bxs_{*})] \nonumber \\ 
         &\leq \frac{1}{\gamma_h\mu}\EE[D_{F}(\tilde{\bx}_{s}, \bxs_*)] + 2\alpha^2 B L m\EE[D_F(\tilde{\bx}_{s}, \bxs_{*})]\nonumber \\ 
         &= \left(\frac{1}{\gamma_h\mu} + 2\alpha^2 B L m \right)\EE[D_{F}(\tilde{\bx}_{s}, \bxs_*)]\nonumber.
    \end{align}
    In the first inequality, we used the fact that
\begin{equation*}
\sum_{k = 0}^{m-1} F(\bx_k) = m \sum_{k = 0}^{m-1} \frac{1}{m} F(\bx_k) = m \sum_{\xi = 0}^{m-1} \frac{1}{m} F\left({\textstyle\sum_{k=0}^{m-1} \delta_{k,\xi}\bx_k}\right) = m \EE[F(\tilde{\bx}_{s+1}) \,\vert\, \Fk_{s,m-1}].
\end{equation*}
It still holds by choosing $\tilde{\bx}_{s+1} = \sum_{k = 0}^{m-1} \frac{1}{m} \bx_k$ , in Algorithm~\myref{algo:srvprox2}, and using Jensen inequality to lower bound $\sum_{k = 0}^{m-1} F(\bx_k)$ by $m F(\tilde{\bx}_{s+1})$. 
We finally have
\begin{align}
    \EE[D_F(\tilde{\bx}_{s+1}, \bxs_{*})] &\leq \left(\alpha\left(1-\alpha A\right)m\right)^{-1} \left(\frac{1}{\gamma_h\mu} + 2\alpha^2 B L m\right)\EE[D_{F}(\tilde{\bx}_{s}, \bxs_*)] \nonumber \\ 
    &= \left(\frac{1}{\gamma_h\mu\alpha\left(1-\alpha A\right)m} + \frac{2\alpha B L }{1-\alpha A}\right)\EE[D_{F}(\tilde{\bx}_{s}, \bxs_*)].\nonumber 
\end{align}
Replacing the constants by their values gives the result. \hfill \qed
\clearpage
\section{{Full results of the numerical experiments}}\label{app:figures}

\begin{figure}[ht]
     \centering
     \begin{subfigure}[b]{0.32\textwidth}
         \centering
         \includegraphics[width=\textwidth]
         {poisson_0.5.png}
     \end{subfigure}
     \begin{subfigure}[b]{0.32\textwidth}
         \centering
         \includegraphics[width=\textwidth]
         {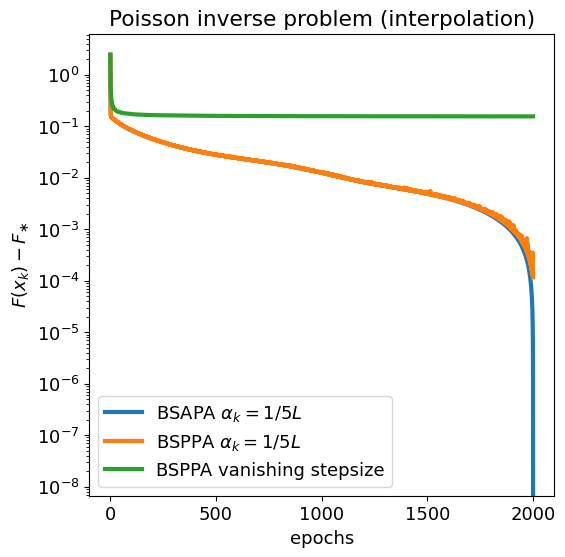}
     \end{subfigure}
     \begin{subfigure}[b]{0.32\textwidth}
         \centering
         \includegraphics[width=\textwidth]
         {poisson_5.png}
     \end{subfigure}
    \begin{subfigure}[b]{0.32\textwidth}
         \centering
         \includegraphics[width=\textwidth]
         {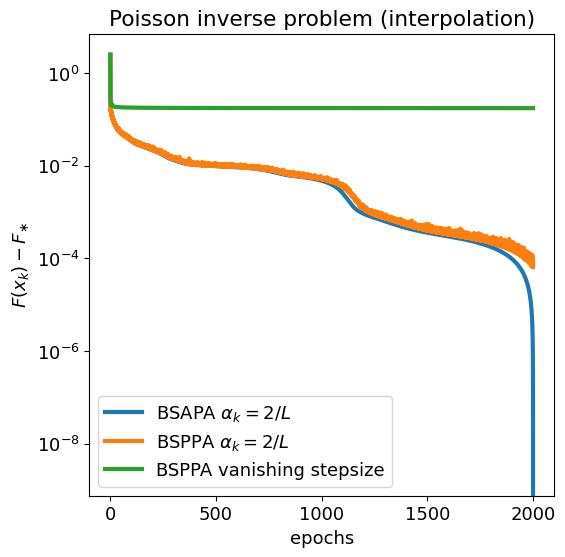}
     \end{subfigure}
    \begin{subfigure}[b]{0.32\textwidth}
         \centering
         \includegraphics[width=\textwidth]
         {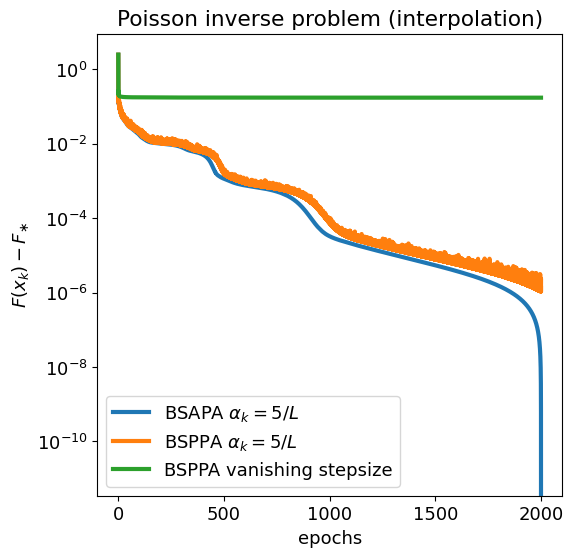}
     \end{subfigure}
    \caption{Poisson linear inverse problem \emph{(interpolation case)} with different stepsizes.
    }
    \label{fig:linear}
\end{figure}
    \begin{figure}[ht]
     \centering
     \begin{subfigure}[b]{0.32\textwidth}
         \centering
         \includegraphics[width=\textwidth]
         {tomo_0.5_5L.png}
     \end{subfigure}
     \begin{subfigure}[b]{0.32\textwidth}
         \centering
         \includegraphics[width=\textwidth]
         {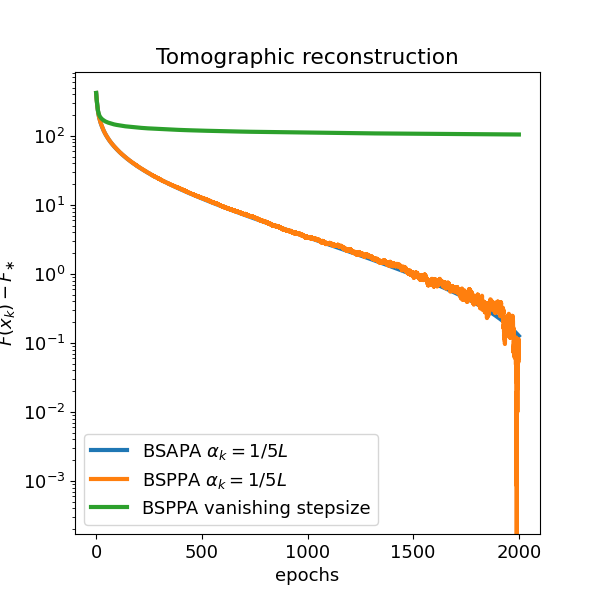}
     \end{subfigure}
     \begin{subfigure}[b]{0.32\textwidth}
         \centering
         \includegraphics[width=\textwidth]
         {tomo_5_5L.png}
     \end{subfigure}
    \begin{subfigure}[b]{0.32\textwidth}
         \centering
         \includegraphics[width=\textwidth]
         {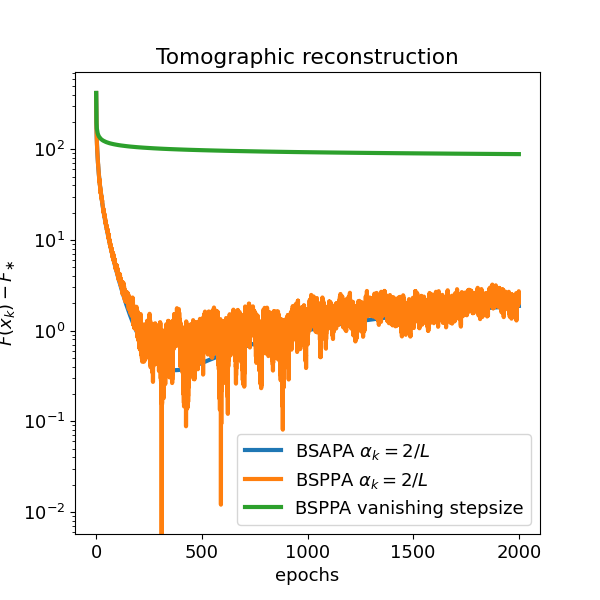}
     \end{subfigure}
    \begin{subfigure}[b]{0.32\textwidth}
         \centering
         \includegraphics[width=\textwidth]
         {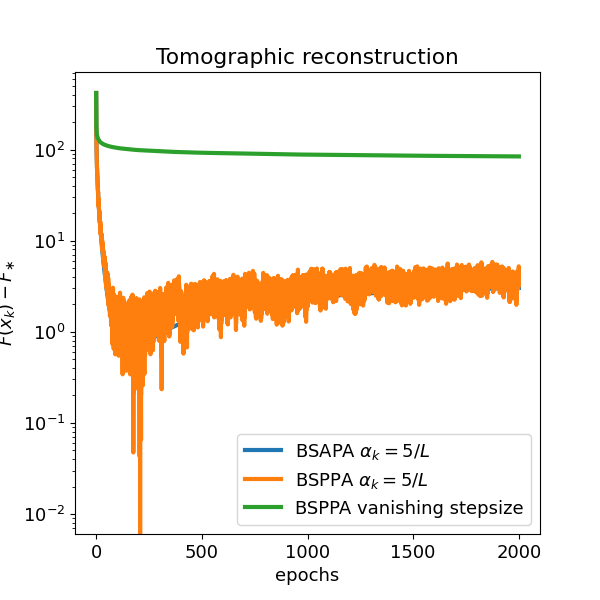}
     \end{subfigure}
    \caption{Tomographic reconstruction with different stepsizes.
    }
    \label{fig:tomo2}
\end{figure}
\begin{figure}[ht]
\centering
     \begin{subfigure}[b]{0.32\textwidth}
         \centering
         \includegraphics[width=\textwidth]
         {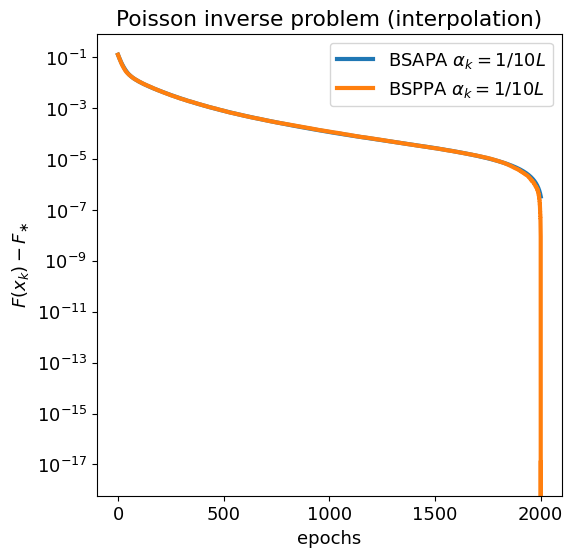}
     \end{subfigure}
     \begin{subfigure}[b]{0.32\textwidth}
         \centering
         \includegraphics[width=\textwidth]
         {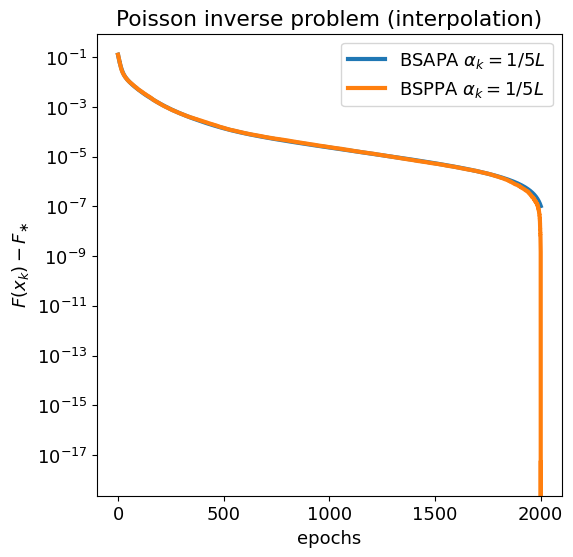}
     \end{subfigure}
    \begin{subfigure}[b]{0.32\textwidth}
         \centering
         \includegraphics[width=\textwidth]
         {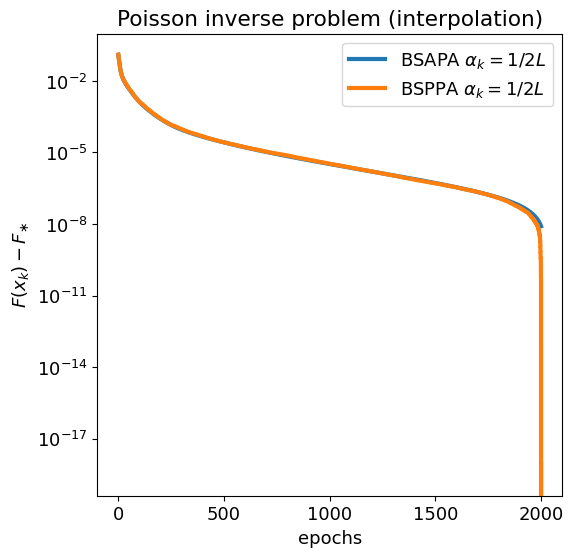}
    \end{subfigure}
    \caption{
    The case of diagonal $A$: closed-form solution to the proximal mapping. There is no instability for BSPPA in that case, showing that its instability in Figure~\myref{fig:tomo} is due to the inexact proximal mapping. This proves it is less stable to inexact algorithms compare to its variance-reduced counterparts.
    }
    \label{fig:tomodiag}
\end{figure}
\clearpage
\bibliography{paper}
\bibliographystyle{plainnat}
\end{document}

%% file: latex_settings.tex
\makeatletter
\if@titlepage%
\else
    \def\abstract@style{heading} 
    \def\abstract@size{\small}
    
    \newcommand{\setabstractstyle}[1]{%
      \def\abstract@style{#1}%
    }
    
    \newcommand{\setabstractsize}[1]{%
      \def\abstract@size{#1}%
    }
    
    \let\orig@abstract\abstract
    \let\endorig@abstract\endabstract
    
    \renewenvironment{abstract}{%
      \if@twocolumn
        \orig@abstract%
      \else
        \abstract@size%
        \edef\temp@a{inline}%
        \ifx\abstract@style\temp@a%
          \quotation
          \noindent\textbf{\abstractname.}\ %
          \ignorespaces
        \else
          \begin{center}%
          {\bfseries\abstractname\vspace{-.5em}\vspace{\z@}}%
          \end{center}%
          \quotation
          \ignorespaces
        \fi
      \fi
    }{%
      \if@twocolumn
        \endorig@abstract
      \else
        \endquotation 
      \fi
    }
\fi
\makeatother
\setabstractstyle{inline} 
\renewcommand{\proofname}{Proof}%
\makeatletter
\renewenvironment{proof}[1][\proofname]{%
  \par\pushQED{\qed}\normalfont
  \topsep6\p@\@plus6\p@\relax 
  \trivlist
  \item[\hskip\labelsep\bfseries #1\@addpunct{.}]\ignorespaces
}{%
  \popQED\endtrivlist\@endpefalse
}
\makeatother
\usepackage{tocloft}
\makeatletter
\newif\ifinappendixbookmark
\inappendixbookmarkfalse

\let\oldappendix\appendix
\renewcommand{\appendix}{%
  \oldappendix
  \inappendixbookmarktrue 
}
\makeatother

\newif\ifdisableappendixbookmark
\disableappendixbookmarkfalse

\makeatletter
\newcommand*{\AddAppendixPrefixInBookmarks}{%
  \ifinappendixbookmark
    \ifnum\bookmarkget{level}=1
      \ifdisableappendixbookmark
      \else
        \preto\bookmark@text{Appendix\space}%
      \fi
    \fi
  \fi
}
\bookmarksetup{
  addtohook=\AddAppendixPrefixInBookmarks,
}
\makeatother
\makeatletter
\renewcommand\Hy@numberline[1]{#1.\space}%
\makeatother

\makeatletter
\newif\if@appendixsec
\@appendixsecfalse
\apptocmd{\appendix}{%
  \@appendixsectrue
}{}{}

\renewcommand{\@seccntformat}[1]{%
  \if@appendixsec
    \ifstrequal{#1}{section}{%
      Appendix~\thesection.\hskip.5em
    }{%
      \csname the#1\endcsname.\hskip.5em
    }%
  \else
    \csname the#1\endcsname.\hskip.5em
  \fi
}
\makeatother
\makeatletter
\newif\ifinappendixtoc
\inappendixtocfalse

\let\oldappendixtoc\appendix
\renewcommand{\appendix}{%
  \oldappendixtoc
  \addtocontents{toc}{\protect\inappendixtoctrue}%
}
\makeatother

\makeatletter
\let\oldnumberline\numberline
\renewcommand{\numberline}[1]{%
  \ifinappendixtoc
    \oldnumberline{\vphantom{[]}#1.\hskip.5em}
  \else
    \oldnumberline{#1.\vphantom{[]}}%
  \fi
}
\makeatother

\makeatletter
\let\oldcontentsline\contentsline

\renewcommand{\contentsline}[4]{%
  \ifinappendixtoc
    \ifstrequal{#1}{section}{%
      \oldcontentsline{#1}{\vphantom{[]}Appendix~#2}{\vphantom{[]}#3}{#4}%
    }{%
      \oldcontentsline{#1}{\vphantom{[]}#2}{\vphantom{[]}#3}{#4}%
    }%
  \else
    \oldcontentsline{#1}{\vphantom{[]}#2}{\vphantom{[]}#3}{#4}%
  \fi
}
\makeatother
\makeatletter
\AtBeginDocument{%
  \pretocmd{\NAT@hyper@}{\vphantom{[]}}{}{}%
}
\makeatother

\let\oldbibliographystyle\bibliographystyle
\renewcommand{\bibliographystyle}[1]{%
    \addtocontents{toc}{\protect\inappendixtocfalse}
    \disableappendixbookmarktrue
    \oldbibliographystyle{#1}%
}

\AtEndDocument{
\addcontentsline{toc}{section}{\refname}
}

\let\LTXlabel\label
\makeatletter
\renewcommand{\footnote}[2][\empty]{%
  \nolinebreak%
  \refstepcounter{footnote}
  \global\edef\sfootnote@arabic{\arabic{footnote}}%
  \global\protected@edef\sfootnote@the{\thefootnote}%
  \ltx@ifpackageloaded{hyperref}{%
    \ifHy@hyperfootnotes%
      \refstepcounter{Hfootnote}
      \global\let\Hy@saved@currentHref\@currentHref%
      \hyper@makecurrent{Hfootnote}%
      \global\let\Hy@footnote@currentHref\@currentHref%
      \global\let\@currentHref\Hy@saved@currentHref%
    \fi%
  }{}%
  \xdef\sfootnote@opt{#1}%
  \ifx\sfootnote@opt\empty%
    \footnotetext{\LTXlabel{fnr:\sfootnote@arabic}#2}%
  \else%
    \ltx@ifpackageloaded{hyperref}{%
      \footnotetext[#1]{\phantomsection\LTXlabel{fnr:\sfootnote@arabic}#2\vphantom{Xg}}%
    }{%
      \footnotetext[#1]{\LTXlabel{fnr:\sfootnote@arabic}#2}%
    }%
  \fi%
  \ltx@ifpackageloaded{hyperref}{%
    \ifHy@hyperfootnotes%
      \hbox{\@textsuperscript{\,\normalfont\ref{fnr:\sfootnote@arabic}\vphantom{Xg}}}
    \else%
      \hbox{\@textsuperscript{\normalfont\ref*{fnr:\sfootnote@arabic}}}%
    \fi%
  }{%
    \hbox{\@textsuperscript{\normalfont\ref{fnr:\sfootnote@arabic}}}%
  }%
  \;%
  \ignorespaces%
}
\makeatother

\makeatletter
\newif\if@maketitle@nohyper
\let\orig@maketitle\maketitle
\renewcommand{\maketitle}{%
  \@maketitle@nohypertrue
  \orig@maketitle
  \@maketitle@nohyperfalse
}
\makeatother

\makeatletter
\renewcommand{\footnotemark}[1][]{%
  \leavevmode
  \ifhmode\edef\@x@sf{\the\spacefactor}\nobreak\fi
  \def\@tempa{#1}%
  \ifx\@tempa\@empty
    \stepcounter{footnote}%
    \protected@xdef\@thefnmark{\thefootnote}%
  \else
    \begingroup
      \c@footnote #1\relax
      \unrestored@protected@xdef\@thefnmark{\thefootnote}%
    \endgroup
  \fi
  \ltx@ifpackageloaded{hyperref}{%
  \ifHy@hyperfootnotes
    \if@maketitle@nohyper
    \else
      \refstepcounter{Hfootnote}%
      \global\let\Hy@saved@currentHref\@currentHref
      \hyper@makecurrent{Hfootnote}%
      \global\let\Hy@footnote@currentHref\@currentHref
      \global\let\@currentHref\Hy@saved@currentHref
    \fi
  \fi
  }{}%
  \textsuperscript{%
    \normalfont
    \ltx@ifpackageloaded{hyperref}{%
      \ifHy@hyperfootnotes
        \if@maketitle@nohyper
          \@thefnmark
        \else
          \,\hyper@linkstart{link}{\Hy@footnote@currentHref}%
          \@thefnmark\vphantom{Xg}%
          \hyper@linkend
        \fi
      \else
        \@thefnmark
      \fi
    }{%
       \@thefnmark
    }%
  }%
  \ifhmode\spacefactor\@x@sf\fi
  \;\ignorespaces
}
\makeatother